\documentclass[a4paper,11pt,reqno,table]{amsart}
\usepackage[margin=3cm]{geometry}
\usepackage[dvipsnames]{xcolor}
\usepackage{graphicx,tikz,calc, amssymb,amsmath,amsthm,amsfonts,mathrsfs, enumerate, enumitem, adjustbox, hyperref}
\usepackage[vcentermath]{youngtab}

\def\Sym{\operatorname{Sym}}
\def\SSYT{\operatorname{SSYT}}

\DeclareMathOperator{\tr}{tr}
\numberwithin{equation}{section}

\newlength\mylength
\def\vlongmapsto#1{%
\begin{tikzpicture}
\draw (0,0.5mm) -- (0,-0.5mm);
\setlength{\mylength}{\widthof{#1}}
\draw[->] (0,0) -- (1.2\mylength,0) node[above,midway] {#1};
\end{tikzpicture}
}

\newcommand{\bwedge}[1]{{\textstyle{\bigwedge\!\!^{\,#1}\, }}}
\newcommand{\Ext}{\mathchoice{{\textstyle\bigwedge}}%
	{{\bigwedge}}%
	{{\textstyle\wedge}}%
	{{\scriptstyle\wedge}}}

\DeclareMathOperator{\ev}{ev}

\DeclareMathOperator{\im}{im}

\newcommand{\E}{\hskip-1pt E}

\newsavebox\MBox

\DeclareMathOperator{\diag}{diag}

\usetikzlibrary{positioning}
\usepackage{amssymb,amsmath,amsthm,amsfonts,mathrsfs}

\def\N{\mathbb N} 
\def\C{\mathbb C}

\def\F{\mathbb F}

\def\SL{\mathrm{SL}}
\def\GL{\mathrm{GL}}

\def\Sym{\operatorname{Sym}}
\def\x{\mathbf{x}}
\def\y{\mathbf{y}}
\def\z{\mathbf{z}}
\def\XX{\mathbf{X}}
\def\YY{\mathbf{Y}}

\def\sgn{\text{sgn}}

\renewcommand{\leq}{\leqslant}
\renewcommand{\geq}{\geqslant}
\renewcommand{\le}{\leqslant}
\renewcommand{\ge}{\geqslant}

\newtheorem{Theorem}{Theorem}[section]
\newtheorem*{theorem*}{Theorem}
\newtheorem{MainCorollary}[Theorem]{Corollary}
\newtheorem{Corollary}[Theorem]{Corollary}
\newtheorem{Proposition}[Theorem]{Proposition}

\newtheorem{Lemma}[Theorem]{Lemma}
\newtheorem{Example}[Theorem]{Example}
\newtheorem{Remark}[Theorem]{Remark}

\theoremstyle{definition}

\newtheorem{Definition}[Theorem]{Definition}

\newcounter{thmlistcnt}
\newenvironment{thmlist}%
{\setcounter{thmlistcnt}{0}%
	\begin{list}{\emph{(\roman{thmlistcnt})}}{%
			\usecounter{thmlistcnt}%
			\setlength{\topsep}{3pt}%
			\setlength{\leftmargin}{0pt}%
			\setlength{\itemsep}{-3pt}%
			\setlength{\labelwidth}{17pt}
			\setlength{\itemindent}{30pt}}%
	}%
	{\end{list}}%

\newenvironment{bulletlist}%
{\begin{list}{$\bullet$}{%
			\usecounter{thmlistcnt}%
			\setlength{\topsep}{3pt}%
			\setlength{\leftmargin}{0pt}%
			\setlength{\itemsep}{2pt}%
			\setlength{\labelwidth}{14pt}
			\setlength{\itemindent}{24pt}}%
	}%
	{\end{list}}%	

\newcommand{\qbinom}[2]{\genfrac{[}{]}{0pt}{}{#1}{#2}_q}

\newcommand{\SymG}{\mathfrak{S}}
\newcommand{\DN}{D}

\newcommand{\deltaD}{a}

\newcommand{\wcirc}[2]{\draw (#1,#2) circle (2pt);\fill[white] (#1,#2) circle (2pt);}

\title[Modular isomorphisms of \texorpdfstring{$\mathrm{SL}_2(\mathbb{F})$}{SL2(F)}-plethysms]{Modular isomorphisms of \texorpdfstring{$\mathrm{SL}_2(\mathbb{F})$}{SL2(F)}-plethysms for Weyl modules labelled by hook partitions}

\author[\'A.~Gutiérrez, \'A.~L.~Mart\'inez, M.~Szwej, M.~Wildon]
{\'Alvaro Guti\'errez, \'Alvaro L. Mart\'inez,\\ Micha\l\ Szwej, and Mark Wildon}
\thanks{ÁG was funded by a University of Bristol Research Training Support Grant.
	MS and MW thank the Heilbronn Institute for Mathematical Research for financial support.
}

\address{University of Bristol}
\email[Á.~Gutiérrez]{a.gutierrezcaceres@bristol.ac.uk}
\email[M.~Szwej]{michal.szwej@bristol.ac.uk}
\email[M.~Wildon]{mark.wildon@bristol.ac.uk}
\address{Columbia University}
\email[Á.~L. Mart\'inez]{alm2297@columbia.edu}

\date{\today} % Bake in before submission

\begin{document}
	\maketitle
	
	\begin{abstract}
		Let $\Delta^\lambda$ be the Weyl functor for the partition $\lambda$ and
		let $E$ be the natural $2$-dimensional
		representation of $\SL_2(\F)$, where $\F$ is an arbitrary field.
		We give an explicit isomorphism showing that any
		$\SL_2(\F)$-plethysm $\Delta^{(M,1^N)}\Sym^d \E$ factors as a tensor product of 
		two simpler $\SL_2(\F)$-plethysms, each defined using only symmetric powers. This result categorifies Stanley's Hook Content Formula for hook-shaped partitions and 
		proves a conjecture of Mart\'inez--Wildon (2024). %í
		In a similar spirit we categorify
		the classical binomial identity $\binom{a}{b}\binom{b}{c}=\binom{a}{c}\binom{a-c}{b-c}$, obtaining
		a new family of $\SL_2(\F)$-isomorphisms between tensor products of plethysms.
		Our methods are characteristic independent and provide a framework that is broadly applicable to the study
		of isomorphisms between plethystic representations of $\SL_2(\F)$.
	\end{abstract}
	
	\thispagestyle{empty}

	\section{Introduction}
	Let $\F$ be an arbitrary field and let $E$ be the natural representation of the special
	linear group $\SL_2(\F)$.
	We define an $\SL_2(\F)$-\emph{plethysm} to be a representation of $\SL_2(\F)$ of the form
	$\Delta^\lambda \Sym^d \E$, where $\Delta^\lambda$ is the Weyl functor for the partition $\lambda$.
	Numerous isomorphisms between $\SL_2(\C)$-plethysms
	are known: see \cite{PagetWildon21} for a comprehensive account. Much less is known about 
	\emph{characteristic-free} isomorphisms, holding over the arbitrary field $\F$.
	In \cite{McDowellWildon}, McDowell and the fourth author generalise the 
	classical Wronskian isomorphism to an explicit characterstic-free isomorphism
	\begin{equation}\label{eq:Wronskian}
		\Sym_N\Sym^{d}\E\cong \Ext^N\Sym^{d+N-1}\E. 
	\end{equation}
	In~\cite{MW24}, the second and fourth authors construct an explicit characteristic-free isomorphism
	\[ \Sym^{N-1}\hskip-1pt\E \otimes \Ext^{N+1} \Sym^{d+1}\hskip-1pt\E \cong \Delta^{(2,1^{N-1})} \Sym^d\E\, .\]
	Conjecture 3.3~in \cite{MW24} proposes a more general isomorphism in which $(2,1^{N-1})$ is replaced with
	an arbitrary hook partition. This is proved by our first main theorem.
	
	\begin{Theorem}[Hook plethysms]\label{thm:hook}
		Let $M$, $d \in \N_0$, and let $N \in \N$ be such that $N\le d+1$. There is an isomorphism of $\SL_2(\F)$-representations
		\[ \Sym_{M}\Sym^{N-1}\E\hskip1pt\otimes\hskip1pt\Sym_{M+N}\Sym^{d-N+1}\E \cong \Delta^{(M+1,1^{N-1})}\Sym^{d}\E. \]
	\end{Theorem} 
	
	This result is notable as the first in the literature giving a tensor factorization of an $\SL_2(\F)$-plethysm for an
	arbitrary hook partition.
	We define the Weyl functor $\Delta^{(M+1,1^{N-1})}$ in \S\ref{subsec:WeylHook} below.
	Lower symmetric and exterior powers are defined in
	\S\ref{subsec:exteriorLowerSymmetricPowers} and upper symmetric
	powers in~\S\ref{subsec:upperSymmetricPowers}.
	%Example~\ref{ex:Sym}
	%illustrates why we need both versions of the symmetric power.
	
	Our proof gives an explicit isomorphism defined over the integers and so over any field. 
	In Corollary~\ref{cor:GL2-isom} we show
	that the isomorphism in Theorem~\ref{thm:hook}
	lifts to an isomorphism of representations of $\GL_2(\F)$ 
	if we take the tensor product
	of the left-hand side with the $\binom{N}{2}$-th power of the determinant representation
	of $\GL_2(\F)$. (A similar `lift' is possible for all the $\SL_2(\F)$-isomorphisms in this paper: see Lemma~\ref{lemma:SL2toGL2}.)
	%(This is expected because the left-hand side
	%as polynomial degree $(M+N)d + N(N-1)$
	%whereas the right-hand side has degree $(M+N)d$.)
	Taking $\F = \C$ and equating the
	characters on either side using~\eqref{eq:qHookCharacter} 
	and its special case~\eqref{eq:qSym} 
	we obtain
	\[  q^{\binom{N}{2}}\qbinom{M+N-1}{M} \qbinom{M+d+1}{M+N} = s_{(M+1,1^{N-1})}(1,q,\ldots, q^d)\]
	where $s_\lambda$ is the Schur function for the partition $\lambda$. 
	While this identity
	can be deduced from Stanley's Hook Content Formula for the partition $(M+1,1^{N-1})$, Theorem~\ref{thm:hook} gives an 
	independent proof. 
	
	Our second main theorem is motivated by the basic identity
	\begin{equation}
		\label{eq:GTL}
		\binom{M+N}{M}\binom{M+N+d}{M+N} = \binom{N+d}{N} \binom{M+N+d}{M}. \end{equation}
	Because each side equals
	$(M+N+d)! \bigl/ M!N! d!$, this identity is sometimes referred to as
	\emph{trinomial revision} \cite[page~174]{GKP89}. 
	It is the decategorification, by taking dimensions, of our second main theorem.
	
	\begin{Theorem}[Trinomial plethysms]%[GT\&L]
		\label{thm:GTL}
		Let $M$, $N$, $d \in \N_0$. There is an isomorphism of $\SL_2(\F)$-representations
		\[
		\Sym_M\Sym^N\hskip-1pt\E\otimes\Sym_{M+N}\Sym^d\E \cong
		\Sym_N\Sym^{d}\E\otimes\Sym_M\Sym^{d+N}\hskip-1pt\E.
		\]      
	\end{Theorem}
	
	While the
	isomorphism
	in Theorem~\ref{thm:GTL} can be proved to exist when $\F = \C$ by character calculations,
	as
	Example~\ref{ex:Sym} 
	shows, upper and lower symmetric powers typically define non-isomorphic
	modules over fields of prime characteristic. It is therefore remarkable
	that, when the symmetric powers are chosen correctly,
	there is an isomorphism holding over an arbitrary field. Again, our proof gives
	it explicitly. 
    In \cite{RSWY}, submitted to the arXiv after the original version of this paper, there is a different approach to studying isomorphisms between plethysms via sheaf (hyper)cohomology. In Theorem 3.1 the authors give an isomorphism between certain complexes which, in each degree, recovers Theorem~\ref{thm:GTL} over $\F=\C$ for parameters satisfying a certain divisibility condition. They conjecture that this divisibility condition can be relaxed. This is proved by our Theorem~\ref{thm:GTL} over an arbitrary field.

        To illustrate the power of the theorem we state the following three special cases.
	
	\begin{MainCorollary}\label{cor:T&L}
		Let $K$, $d \in \N_0$. There exist isomorphisms of $\SL_2(\F)$-repres\-entations:
		\begin{thmlist}
			\item 
			\label{thm:T&L1}
			$
			\Sym^{d+K}\E\otimes\Sym_K\Sym^{d}\E\, \cong
			\Sym^K\E\otimes\Sym_{K+1}\Sym^d\E\,;
			$
			\medskip
			\item
			\label{thm:T&L2}
			$
			\Sym_K\E\otimes\Sym_{K+1}\Sym^{d}\E \cong \Sym^{d}\E\otimes\Sym_K\Sym^{d+1}\E\,;
			$
			\medskip
			\item
			\label{thm:T&L3}
			$
			\Sym_{d+K}\E\otimes\Sym_{K}\Sym^{d}\E \cong \Sym_{d}\E\otimes\Sym_K\Sym^{d+1}\E.
			$
			\medskip
		\end{thmlist}
	\end{MainCorollary}
	
	\vspace*{-6pt}
	These isomorphisms are of 
	independent interest. They decategorify to the \emph{team-and-leader} identities in~\eqref{eq:TLidentities}.
	As an extension, we introduce \emph{election processes} on team hierarchies with $k$ layers, which are the possible expressions arising from $\binom{N_k}{N_{k-1}}\dots\binom{N_3}{N_2}\binom{N_{2}}{N_1}$ upon successive applications of the trinomial revision identity~\eqref{eq:GTL}. We finish by showing that the number of election processes on team hierarchies with $k$ layers is given by the Catalan number $C_{k-1}$.
	% We finish by showing that there are precisely $C_{k-1}$
	% expressions arising from $\binom{N_1}{N_2}\binom{N_2}{N_3}\ldots\binom{N_{k-1}}{N_k}$ upon successive applications of the trinomial revision identity \eqref{eq:GTL}, where $C_k$ is the
	% $k^\mathrm{th}$ Catalan number.
	Lifting the obtained expressions with Theorem \ref{thm:GTL} gives in turn~$C_{k-1}$ pairwise isomorphic $\SL_2(\F)$-plethysms. Yet again, these isomorphisms are far from obvious.
	
	We prove our $\SL_2(\F)$-isomorphisms by constructing
	an explicit model for the plethysms $\Delta^\lambda \Sym^d \E$ as
	subrepresentations of suitable polynomial algebras;
	these include the algebra of symmetric functions defined over $\F$:
	see~\eqref{eq:Lambda}.
	An important motivation for this model was Grinberg's proof \cite{Grinberg}
	of the Wronskian isomorphism~\eqref{eq:Wronskian},
	in which the map from the left-hand side to the right-hand side
	is described as multiplication by a Vandermonde determinant.
	We believe our model will be of general use when investigating
	plethystic isomorphisms. 
	Thus, while Theorem~\ref{thm:hook}
	proves a conjecture from~\cite{MW24}, the methods used in this paper are entirely novel,
	and of significant independent interest in their own right.
	
	We conclude by mentioning one natural question raised by Theorem~\ref{thm:hook}:
	\emph{what other instances of Stanley's Hook Content Formula
		have characteristic-free modular lifts}? We hope to address
	this in a sequel to this paper.

	\section{Preliminaries}
	\subsection{Basic definitions}
	A \emph{partition} $\lambda = (\lambda_1, \lambda_2, \ldots)$ is a weakly decreasing finite sequence of non-negative integers, ending with infinitely many zeros.
	If $\ell$ is maximal such that $\lambda_\ell \not= 0$ then we say that $\lambda$
	has \emph{length} $\ell$ and write $\lambda$ as $(\lambda_1,\ldots,\lambda_\ell)$.
	The \emph{Young diagram} of $\lambda$ is
	the set $[\lambda] = \{(i, j) : 1 \le i \le \ell, 1 \le j \le \lambda_i \}$ of \emph{boxes}.
	A \emph{tableau} 
	is a filling of the boxes of $[\lambda]$
	with entries from $\N_0$; a tableau is \emph{semistandard}
	if its rows are weakly increasing and its columns are strictly increasing.
	For examples see \S\ref{subsec:WeylHook} below.

	\subsection{Exterior powers and lower symmetric powers}\label{subsec:exteriorLowerSymmetricPowers}
	Since we later apply these functors to 
	the upper symmetric powers $\Sym^d\E$ of the natural $2$-dimensional
	representation~$E$ of $\SL_2(\F)$, as defined in \S\ref{subsec:upperSymmetricPowers} below,
	it is most convenient to take a 
	vector space of dimension $d+1$. Let $V$ be an $\F$-vector space with basis $v_0,v_1,\ldots, v_d$. %; thus $\dim V = d+1$.
	Let $R \in \N_0$ and let $c_1, \ldots, c_R \in \{0,1,\ldots, d\}$.
	The symmetric group $\SymG_R$ acts on elements in the canonical
	basis of $V^{\otimes R}$ by position permutation:
	\[ \sigma \cdot (v_{c_1} \otimes \cdots \otimes v_{c_R})
	= v_{c_{\sigma^{-1}(1)}} \otimes \cdots \otimes v_{c_{\sigma^{-1}(R)}}. \]
	(Note the inverse is correct: $v_{c_i}$ is found in position $\sigma(i)$
	on the right-hand side.) This action extends linearly to an action of $\SymG_R$ on
	$V^{\otimes R}$. 
    
    We define the lower symmetric powers by
	\begin{align*}
		\Sym_R V  
		&= \{ w \in V^{\otimes R} : \sigma \cdot w = w \text{ for all $\sigma \in \SymG_R$} \}.
    \end{align*}

	If $H = \{ \sigma \in \SymG_R : \sigma \cdot (v_{c_1} \otimes \cdots \otimes v_{c_R})
	= v_{c_1} \otimes \cdots \otimes v_{c_R} \} $ then
	we define
	\begin{equation} v_{(c_1,\ldots, c_R)} = \sum_\sigma \sigma \cdot (v_{c_1} \otimes \cdots \otimes
		v_{c_R}) \in \Sym_R V \label{eq:symBasis} \end{equation}
	where the sum is over a set of coset representatives for the cosets $\SymG_R/H$.
	It is clear that
	$\Sym_R V$ is spanned by these symmetric elements, and that
	$v_{(c_1,\ldots, c_R)} = v_{(c'_1,\ldots, c'_R)}$ if $(c_1,\ldots, c_R)$
	and $(c'_1,\ldots, c'_R)$ are equal up to the order of the entries.
	We therefore obtain a basis of $\Sym_R V$ by taking $c_1 \le \ldots \le c_R$.
	Note that it is essential to take coset representatives: for instance
	the symmetrization of $v_1 \otimes v_1 \otimes v_2$ by all $6$ permutations
	in $\SymG_3$ is zero if $\F$ has characteristic $2$.

    Similarly, the exterior powers $\bwedge{R} V$ are defined as the span of the elements
    \[
    v_{c_1} \wedge \cdots \wedge v_{c_R} = \sum_{\sigma \in \SymG_R}
		\sgn(\sigma) \sigma \cdot (v_{c_1} \otimes \cdots \otimes v_{c_R}).
    \]
   Moreover we obtain a basis if we require $c_1 < \ldots < c_R$.

    Note that $\SymG_0$ is trivial and so $V^{\otimes 0}\cong\Sym_0(V)\cong\bwedge{0}V\cong\F$.
	
	\subsection{Weyl functors for hook partitions}\label{subsec:WeylHook}
	We now present a simple explicit construction
	of the Weyl functors $\Delta^{(M+1,1^{N-1})}$ labelled by hook partitions.
	Given a tableau having entries
	$a_0, a_1, \ldots, a_M$ in its top row
	and $a_0, b_1, b_2, \ldots, b_{N-1}$ in its first column,
	define
	$F_\Delta(t) \in \bwedge{N} V \otimes \Sym_M V$
	by
	\begin{equation}
		F_\Delta(t) = \sum (v_{c_0} \wedge v_{b_1} \wedge \cdots \wedge v_{b_{N-1}})
		\otimes v_{c_1} \otimes \cdots \otimes v_{c_{M}} \label{eq:FDelta}
	\end{equation}
	where the sum is over all distinct tuples $(c_0, c_1, \ldots, c_M)$
	obtained by permuting the tuple $(a_0,a_1,\ldots, a_M)$.
	For example, if $M=3$, and $N=3$, and
	$t$ is the tableau
	\[ \young(0225,2,4) \]
	then $(a_0,a_1,a_2,a_3) = (0,2,2,5)$ and so there are $12$ summands in the sum
	defining $F_\Delta(t)$, but 
	all those for which $c_0 = 2$ cancel, because $v_2 \wedge v_2 \wedge v_4 = 0$. Thus
	\begin{align*}
		F_\Delta(t) &= (v_0 \wedge v_2 \wedge v_4) \otimes (v_2 \otimes v_2 \otimes v_5
		+ v_2 \otimes v_5 \otimes v_2 + v_5 \otimes v_2 \otimes v_2)
		\\ &+ (v_5 \wedge v_2 \wedge v_4) \otimes (v_0 \otimes v_2 \otimes v_2 + 
		v_2 \otimes v_0 \otimes v_2 + v_0 \otimes v_0 \otimes v_2) .\end{align*}
	%\\ &= (v_0 \wedge v_2 \wedge v_4) \otimes s_{(2,2,5)} +
	%(v_2 \wedge v_4 \wedge v_5) \otimes s_{(0,2,2)} \end{align*}
Expressed 
in the canonical basis of $\bwedge{3} V \otimes \Sym_3 V$
this is $(v_0 \wedge v_2 \wedge v_4) \otimes v_{(2,2,5)} +
(v_2 \wedge v_4 \wedge v_5) \otimes v_{(0,2,2)}$.

\subsubsection*{Semistandard basis}
Let $\SSYT_{\le d}(M+1,1^{N-1})$  
denote the set of semistandard tableaux of shape $(M+1,1^{N-1})$ whose
entries lie in $\{0,1,\ldots, d\}$. 
By either
\cite[Theorem II.3.16]{AkinBuchsbaumWeyman}, \cite[Theorem~6.1]{McDowellWeyl}
or \cite[Proposition~3.13]{McDowellThesis},
the Weyl module $\Delta^{(M+1,1^{N-1})}(V)$ is 
the subrepresentation of
$\bwedge{N} V \otimes V^{\otimes M}$ 
with basis all $F_\Delta(t)$ for $t \in \SSYT_{\le d}(M+1, 1^{N-1})$.

%\begin{center}
%\begin{tikzpicture}[x=0.7cm, y=-0.6cm]
%\draw (0,0)--(3.5,0); \draw (0,1)--(3.5,1); 
%\draw (0,0)--(0,3.5); \draw (1,0)--(1,3.5);
%\draw (1,0)--(1,1); \draw (2,0)--(2,1); \draw (3,0)--(3,1);
%\node at (0.5, 0.5) {$1$}; \node at (1.5, 0.5) {$a_1$}; \node at (2.5, 0.5) {$ a_2$};
%\node at (0.5, 1.5) {$ b_1$}; \node at (0.5, 2.5) {$b_2$};
%\draw (0,1)--(1,1); \draw (0,2)--(1,2); \draw (0,3)--(1,3);
%\node at (4.5,0.5) {$\ldots$};
%\draw (5.5,0)--(7,0); \draw (5.5,1)--(7,1);
%\draw (6,0)--(6,1); \draw (7,0)--(7,1);
%\node at (6.5,0.5) {$ a_M$};
%\draw (0,5.5)--(0,7); \draw (1,5.5)--(1,7);
%\draw (0,6)--(1,6); \draw (0,7)--(1,7);
%\node at (0.5,6.5) {$\scriptstyle b_{ N\!-\!1}$};
%
%\end{tikzpicture}
%
%\end{center}

\begin{Example}[Lower symmetric and exterior powers]
	It follows immediately from the definitions in \S\ref{subsec:exteriorLowerSymmetricPowers}
	and~\eqref{eq:FDelta} that $\Delta^{(1^R)}(V) = \bwedge{R} V$
	and $\Delta^{(R)}(V) = \Sym_R V$.
\end{Example}

\begin{Example}
	If $M=1, N=2$, and $d=2$, then $\Delta^{(2,1)}(V)$ 
	has as a basis all $F_\Delta(t)$ for $t$ one of the eight semistandard tableaux
	\[ \young(00,1)\, , \ \young(01,1)\, , \ \young(00,2)\, , \ \young(01,2)\, , 
	\ \young(02,1)\, , \ \young(02,2)\, , \ \young(11,2)\, , \ \young(12,2)\, . \]
	When $\F = \C$, the restriction of $\Delta^{(2,1)}(V)$ to the special
	unitary group $\mathrm{SU}_3(\C)$ is the famous eight-fold way representation (see \cite[page 179]{FH13}).
	If instead~$\F$ has characteristic~$3$ then 
	\begin{align*} F_\Delta\bigl( \, \young(02,1) \, \bigr) - 
		F_\Delta\bigl( \, \young(01,2) \, \bigr) &=
		(v_0 \wedge v_1) \otimes v_2 + (v_1 \wedge v_2) \otimes v_0 
		+ (v_2 \wedge v_0) \otimes v_1 
		\\ &= \sum_{\sigma \in \SymG_3} \mathrm{sgn}(\sigma) \sigma \cdot (v_0 \otimes v_1 \otimes v_2).
	\end{align*}
	Since the span of this vector affords the $1$-dimensional
	determinant representation of $\GL(V)$,
	in this case, $\Delta^{(2,1)}(V)$ is  reducible.
\end{Example}

\subsubsection*{Character}
Suppose that the tableau $t$ has exactly $a_i(t)$ entries equal to $i$ for each
$i \in \{0,1,\ldots, d\}$. Then from~\eqref{eq:FDelta} 
the action of the diagonal matrix $\diag(\gamma_0,\gamma_1,\ldots, \gamma_d)$ is given by
\[ \diag(\gamma_0,\gamma_1,\ldots, \gamma_d)F_\Delta(t) =
\gamma_0^{a_0(t)} \gamma_1^{a_1(t)} \hskip-1pt
\ldots \gamma_d^{a_d(t)}F_\Delta(t).\]
Thus each $F_\Delta(t)$ is a simultaneous eigenvector
for the subgroup of $\GL(V)$ of diagonal matrices (in our chosen basis).
The formal character
of $\Delta^{(M+1,1^{N-1})}(V)$ is therefore the Schur  polynomial $s_{(M+1,1^{N-1})}$ 
in the variables $x_0,x_1,\ldots, x_{d}$:
\begin{equation}
	\label{eq:hookCharacter} s_{(M+1,1^{N-1})}(x_0,x_1,\ldots, x_d) = \!
	\sum_{t \in \SSYT_{\le d}(M+1,1^{N-1})}\!\! x_0^{a_0(t)}
	x_1^{a_1(t)} \hskip-1pt\ldots x_d^{a_d(t)}. \end{equation}
Later, in \S\ref{subsec:symmetric-polys}, we present an algebraic definition of the Schur polynomial $s_\lambda$ for a general partition $\lambda$. For the equivalence of the two definitions
(in the case $\F = \C$)
see \cite[Theorem~7.15.1]{StanleyEC2}.

\subsubsection*{Dimension}
The following proposition on the dimension of $\Delta^{(M+1,1^{N-1})}(V)$
can also be proved by setting $q=1$ in Stanley's Hook Content Formula 
\cite[Corollary 7.21.4]{StanleyEC2}, or, with more work, from Weyl's Dimension Formula.
We give a proof to make this article self-contained.
Recall that $\dim V = d+1$. 

\begin{Proposition}\label{prop:dim-delta}
	We have $\dim \Delta^{(M+1,1^{N-1})}(V) = \binom{d+M+1}{M+N} \binom{M+N-1}{M}$.
\end{Proposition}

\begin{proof}
	We have seen that the dimension is the number of 
	semistandard tableaux of shape $(M+1,1^{N-1})$ having entries
	from $\{0,1\ldots, d\}$. Keeping the existing notation, suppose that
	$t$ has $a_0 \le a_1 \le \ldots \le a_M$ in its first row and $b_0< b_1 < \ldots< b_{N-1}$
	in its first column. 
	For each $i \in \{1,\ldots, N-1\}$ let
	$s_i = | \{j \in \{1,\ldots, M\} : a_j < b_i \} |$.
	Note that $s_1 \le \ldots \le s_{N-1} \le M$.
	We now show that 
	\[ S = \{b_1+s_1,\ldots, b_{N-1}+s_{N-1} \} \cup \{a_0, a_1 + 1 , \ldots, a_M + M \} \]
	has $M+N$ distinct elements. In the two cases below, $\{a_1,\ldots a_j\}$ should be read
	as a multiset:
	\begin{bulletlist}
		\item if $a_j < b_i$ then
		$s_i \ge |\{a_1,\ldots, a_j\}| = j$, and so $j \le s_i$ and $a_j + j < b_i + s_i$;
		\item if $a_j \ge b_i$ then
		$s_i < |\{a_1,\ldots, a_j \}| = j$, and so $j > s_i$ and $a_j + j > b_i + s_i$.
	\end{bulletlist}
	The maximum value in $S$
	is either $b_{N-1} + s_{N-1} \le d + M$, or $a_M + M \le d+M$.
	It follows that $S$ is an $(M+N)$-subset
	of $\{0,1,\ldots, d+M\}$.
	It 
	can be chosen in $\binom{M+d+1}{M+N}$ ways. Moreover,
	the tableau $t$ is uniquely determined by $S$ and its $M$-subset $\{a_1+1,\ldots, a_M + M\}$.
	Since this $M$-subset can be freely chosen from $S \backslash \{\min S\}$
	and $|S \backslash \{\min S\}| = M+N-1$,
	it follows that there are precisely $\binom{d+M+1}{M+N} \binom{M+N-1}{M}$
	semistandard tableaux.
\end{proof}

\subsection{The multiplication map: hook Weyl modules as images}
Each $F_\Delta(t)$ is defined in~\eqref{eq:FDelta}
by symmetrization over the entire top row in the
tableau $t$. Hence each $F_\Delta(t)$ is symmetric
with respect to the final $M$ tensor positions in $V^{\otimes (M+N)}$. Hence
\begin{equation} F_\Delta(t) \in \bwedge{N} \hskip-1pt V \otimes \Sym_M\! V %\label{eq:FDeltaWedgeSym} 
\end{equation}
where the right-hand side is a subspace of $\bwedge{N} \hskip-1pt V \otimes V^{\otimes M}$.
Recall that $V$ has the chosen basis $v_0,v_1,\ldots, v_d$.
For $M \in \N$ let
\begin{equation}\label{eq:muM}
	\mu_M : \bwedge{N} \hskip-1pt V \otimes V^{\otimes M} \rightarrow \bwedge{N+1} \hskip-1pt V \otimes V^{\otimes M-1}
\end{equation}
be the  map defined by linear extension of
\[\begin{split} v_{c_0} \wedge v_{b_1} \wedge \cdots \wedge v_{b_{N-1}}
	&\otimes v_{c_1} \otimes v_{c_2}\otimes \cdots \otimes v_{c_M} \\
	&\vlongmapsto{$\scriptstyle \mu_M$} v_{c_0} \wedge v_{b_1} \wedge \cdots \wedge v_{b_{N-1}} 
	\wedge v_{c_1} \otimes v_{c_2} \otimes \cdots \otimes v_{c_M}\end{split}
\]  
(with $\mu_0$ being the zero map) and let $\delta_M$ denote the restriction of $\mu_M$ to
$\bwedge{N} \hskip-1pt V \otimes \Sym_M V$. 
Since the final $M-1$ tensor factors $v_{c_2} \otimes \cdots \otimes v_{c_M}$ 
are the same on each side of the defining equation above, it is clear
that $\delta_M$ has image contained in   $\bwedge{N+1} \hskip-1pt V \otimes \Sym_{M-1}V$.
Thus the image of $\delta_M$ is symmetric under position permutation
in the $M-1$ positions $N+2,\ldots, N+M$.
But it is clear from the definition
of $\mu_M$ that the image of $\delta_M$
is also symmetric under swapping positions $1$ and $N+1$. (These positions have $v_{c_0}$
and $v_{c_1}$ in the right-hand side above.)
Therefore,
using the canonical basis element $v_{(a_1,\ldots, a_M)}$ defined
in~\eqref{eq:symBasis}, we have
\[ \delta_M (v_{b_1} \wedge \cdots \wedge v_{b_N} \otimes v_{(a_1,\ldots, a_M)} )
= F_\Delta(\tilde{t}) \]
where $\tilde{t}$ is the tableau of shape $(M,1^N)$ having first row $a_1,\ldots, a_M$ 
read left to right and first column $a_1, b_1, \ldots, b_N$ read top to bottom. We conclude that $\im \delta_M = \Delta^{(M,1^N)}(V)$.

\subsection{Hook Weyl modules as kernels}
The following lemma 
is known to experts: for instance, it follows from
(2.1) of \cite{MaliakisStergiopoulou}, where
the authors use divided symmetric powers
to show that the maps $\delta_M$ define a chain complex
dual to a suitable partially symmetrized Koszul complex.
We give an elementary self-contained proof, including full details
in a routine calculation to save the reader some effort.

\begin{Lemma}\label{lemma:DeltaKernel}
	We have $\ker \delta_M =   
	\Delta^{(M+1,1^{N-1})}(V)$ for each $M \in \N_0$.
\end{Lemma}  
\begin{proof}
	Note that, by its definition by symmetrization over the 
	top row of a tableau, $\Delta^{(M+1,1^{N-1})}(V)$ 
	is a subspace of the span
	of tensors of the two forms
	\[v_{e} \wedge v_{b_1} \wedge \cdots \wedge v_{b_{N-1}} \otimes v_{e'} \otimes v_{c_2}\otimes \dots
	\otimes  v_{c_M}
	+     v_{e'} \wedge v_{b_1} \wedge \cdots \wedge v_{b_{N-1}} \otimes v_{e} \otimes v_{c_2} \otimes \dots
	\otimes v_{c_M} \]
	for $e\not= e'$
	and   $v_{e} \wedge v_{b_1} \wedge \cdots \wedge v_{b_{N-1}} \otimes v_{e} \otimes v_{c_2} \otimes \dots
	\otimes v_{c_M}$.
	Both of these vectors are in the kernel of 
	$\delta_M$. Therefore we have
	\begin{equation}\label{eq:deltaKernel} \ker \delta_M \supseteq  \Delta^{(M+1,1^{N-1})}(V). \end{equation}
	We have already observed that $\im \delta_M = \Delta^{(M,1^N)}(V)$.
	By Proposition~\ref{prop:dim-delta} and an instance of~\eqref{eq:GTL} to get the third
	equality, we have
	\begin{align*} \dim{} &\Delta^{(M+1,1^{N-1})}(V) + \dim  \Delta^{(M,1^N)}(V) \\
		&= \binom{d+M+1}{M+N}\binom{M+N-1}{M} + \binom{d+M}{M+N}\binom{M+N-1}{M-1} \\
		&= \frac{d+M+1}{M+N} \binom{d+M}{M+N-1}\binom{M+N-1}{M} + \frac{M}{M+N}
		\binom{d+M}{M+N}\binom{M+N}{M} \\
		&= \frac{d+M+1}{M+N} \binom{d+M}{M} \binom{d}{N-1} +
		\frac{M}{M+N}\binom{d+M}{M} \binom{d}{N} \\
		&= \frac{ (d+M+1) N+ M(d-N+1) }{(M+N)(d+1)} \binom{d+M}{M} \binom{d+1}{N} \\
		&= \binom{d+M}{M} \binom{d+1}{N}  \\
		&= \dim \bigl(\bwedge{N} V \otimes \Sym_M V\bigr)
	\end{align*}
	where the final equality holds 
	because by~\eqref{eq:symBasis}, $\Sym_M V$ has a basis indexed
	by the $M$-multisets of $\{0,1,\ldots, d\}$, of which there are $\binom{d+M}{M}$;
	see~\eqref{eq:qSym}.
	By rank-nullity 
	\[ \dim \ker \delta_M + \dim \im \delta_M = \dim \bigl(\bwedge{N} V \otimes \Sym_M V\bigr).
	\]
	Since \smash{$\im \delta_M = \Delta^{(M,1^N)}(V)$}
	it follows that $\dim \ker\delta_M = \dim \Delta^{(M+1,1^{N-1})}(V
	)$
	and so equality holds in~\eqref{eq:deltaKernel}.
\end{proof}

%\begin{Proposition}\label{prop:Weyl-functor}
%	We have
%	\[                 \Delta^{(M+1,1^{N-1})}(V)=\bigl(\bwedge{N} V\otimes\Sym_MV\bigr)
%	\cap \ker \mu_M.
%	\]
%\end{Proposition}
%
%\begin{proof}
%	By definition $\delta_M$
%	is the restriction of $\mu_M$ to $\bwedge{N} V\otimes\Sym_MV$ and
%	by Lemma~\ref{lemma:DeltaKernel}, $\Delta^{(M+1,1^{N-1})}$
%	is the kernel of $\delta_M$.
%\end{proof}

\subsection{Upper symmetric powers}\label{subsec:upperSymmetricPowers}
Fixing a  basis $X$, $Y$ of $E$ we identify
$\Sym^d \E = \langle X^d, X^{d-1}Y ,\ldots, Y^d \rangle$ with degree $d$
homogeneous polynomials in the variables $X$ and~$Y$. The action of $\GL_2(\F)$ and $\SL_2(\F)$ is given explicitly by
\begin{align}\label{action-of-SL_2}
	\begin{pmatrix}
		\alpha & \beta \\
		\gamma & \delta
	\end{pmatrix}\cdot P(X,Y)=P(\alpha X+\gamma Y, \beta X+\delta Y).
\end{align}
Thus the matrix above acts as the unique algebra automorphism
satisfying $X \mapsto \alpha X + \gamma Y$ and $Y \mapsto \beta X + \delta Y$.

\begin{Example}\label{ex:Sym}
We pause to give an example illustrating that the
existence of modular isomorphisms, such as those in our two main theorems,
is a subtle question over fields of prime characteristic.
	The actions of $\GL_2(\F)$ on $\Sym^2 \E$ and $\Sym_2 \E$ are given by 
	the explicit homomorphisms below.
    %\textcolor{magenta}{[MS: aren't the two matrices the other way around?]}
	%\begin{align*}
	\[ \begin{pmatrix}
		\alpha & \beta \\
		\gamma & \delta
	\end{pmatrix} \mapsto
	\bordermatrix{ & \scriptstyle X^2 &\scriptstyle Y^2 &\scriptstyle XY \cr
		& \alpha^2 & \beta^2 & \alpha\beta \cr & \gamma^2 & \delta^2 & \gamma\delta \cr & 2\alpha\gamma & 2\beta\delta & \alpha\delta + \beta\gamma } ,\ \ 
	\begin{pmatrix}
		\alpha & \beta \\
		\gamma & \delta
	\end{pmatrix} \mapsto
	\bordermatrix{ &\scriptstyle X\otimes X & \scriptstyle Y \otimes Y & \scriptstyle X \otimes Y + Y \otimes X \cr
		& \alpha^2 & \beta^2 & 2\alpha\beta \cr & \gamma^2 & \delta^2 & 2\gamma\delta \cr & \alpha\gamma & \beta\delta & \alpha\delta + \beta\gamma }\]
	% \end{align*}   
If $\F$ has characteristic $2$ then
$\langle X^2, Y^2 \rangle \subseteq \Sym^2 E$
is the unique non-trivial proper submodule of $\Sym^2 E$; the quotient by this submodule
is isomorphic to the determinant representation.
Dually, $\Sym_2 E$ has the same composition factors, but in the opposite
order. Since each module is indecomposable, they are
non-isomorphic. We leave it to the reader to verify
that, as representations of $\SL_2(\F)$, we have $(\Sym_2 E)^\star \cong \Sym^2 E$ \cite[\S2.2]{McDowellWildon}.
\end{Example}

\subsection{$\SL_2(\F)$-plethysms}\label{subsec:plethysm}
In the previous subsection we saw that $\Sym^d\E$ has
ordered basis $X^d$, $X^{d-1} Y, \ldots, Y^d$. 
Represented in this basis, the 
diagonal matrix $\diag(1,q)\in\GL_2(\F)$ acts on $\Sym^d \E$ as the diagonal
matrix $\diag(1,q,\ldots, q^{d})$. Setting $x_i = q^i$ in~\eqref{eq:hookCharacter}
we obtain the $\GL_2(\F)$ character of $\Delta^{(M+1,1^{N-1})} \Sym^d \E$:
\begin{equation}
\label{eq:qHookCharacter} 
\tr_{\Delta^{(M+1,1^{N-1})} \Sym^d \E} \mathrm{diag}(1,q) = s_{(M+1,1^{N-1})}(1,q,\ldots, q^{d}).
\end{equation}
Since $s_{(d)}(1,q) = 1 + q + \cdots + q^{d}$, the right-hand side
is the plethysm product of Schur functions
$s_{(M+1,1^{N-1})} \circ s_{(d)}$, evaluated at $1$ and $q$. Similar reasoning follows for the action of the diagonal matrix $\diag(q^{-1}, q)\in\SL_2(\F)$, which justifies our term
`$\SL_2(\F)$-plethysm' for $\SL_2(\F)$-representations
of the form \smash{$\Delta^{(M+1,1^{N-1})} \Sym^d \E$}.
More generally, composition of Weyl functors corresponds
to the plethysm product on general Schur functions:
see for instance \cite[Appendix~A]{Macdonald}. 

\subsection{$q$-binomial coefficients}
Let $[n]_q = (q^n-1)/(q-1) = 1+ q + \cdots + q^{n-1}$;
note that $[d+1]_q$ is the $\GL_2(\C)$-character of $\Sym^d \E$ on $\mathrm{diag}(1,q)$
%and $q^{-\binom{d}{2}}[d+1]_q$ is the $\SL_2(\C)$-character of $\Sym^d \E$.
and $q^{-d}[d+1]_{q^2}$ is the $\SL_2(\C)$-character of $\Sym^d \E$ on
$\mathrm{diag}(q^{-1}, q)$.
Let $[n]_q! = [n]_q [n-1]_q \cdots [1]_q$.
We define the $q$-\emph{binomial coefficient} $\qbinom{n}{m}$ by
\begin{equation}%\label{eq:qBinom}
\qbinom{n}{m} = \frac{[n]_q!}{[m]_q![n-m]_q!  }
\end{equation}
for $0 \le m \le n$. Observe that $\qbinom{n}{m}$ specializes to the binomial
coefficient $\binom{n}{m}$ on setting $q=1$.
It is well known that \smash{$\qbinom{n}{m}$} enumerates
partitions in the $(n-m) \times m$ rectangle by their size,
or equivalently, multisets of $\{0,1,\ldots, m\}$ of size $n-m$
by their sum of entries
(this follows easily from \cite[Proposition 7.8.3]{StanleyEC2}).
Thus 
\begin{align} \qbinom{m+d}{m} &= s_{(m)}(1,q,\ldots, q^d).\label{eq:qSym}
\intertext{Similarly $q^{\binom{m}{2}} \qbinom{d+1}{m}$ enumerates
	$m$-subsets of $\{0,1,\ldots, d\}$ by their sum of entries (see \cite[Proposition 1.3.19]{StanleyEC1})
	and so, using that $s_{(1^m)}$ is the elementary symmetric function of degree $m$,
	we obtain}
q^{\binom{m}{2}} \qbinom{d+1}{m} &= s_{(1^m)}(1,q,\ldots, q^d)\label{eq:qWedge} 
.\end{align}
In particular, by setting $q=1$ in~\eqref{eq:qHookCharacter}
and specializing $M$ and $N$ appropriately,~\eqref{eq:qSym} and~\eqref{eq:qWedge} imply that
%\begin{align}
\begin{equation}
\label{eq:dim-plet} %this was prop:dim-plet, but it's working better for me to make
% it an equation, as it is really quite basic, and maybe not a proposition in its own right
\dim \Sym_a\Sym^b\E = \binom{a+b}{a}, \quad
\dim \bwedge{a} \!\Sym^b \E = \binom{b+1}{a}.
\end{equation}
%\end{align}
Of course, these formulae can also be proved directly by counting the $a$-multisubsets
and the $a$-subsets of $\{0,1,\ldots, b\}$, respectively.

\subsection{Lifting $\SL_2(\F)$-isomorphisms to $\GL_2(\F)$-isomorphisms}
The following result is basic: see for instance \cite[Lemma 3.5]{PagetWildon21}
for a special case. We refer the reader to \cite{Green} for background on polynomial
representations. All we need for our purposes is that
$\Delta^\lambda \Sym^d \E$ has polynomial degree $|\lambda| d$.

\begin{Lemma}\label{lemma:SL2toGL2}
Let $\F$ be an infinite field and let $V$ and $W$ be polynomial representations
of $\GL_2(\C)$ of equal degrees.
If $V \cong_{\SL_2(\F)} W$
then $V \cong_{\GL_2(\F)} W$.
\end{Lemma}

\begin{proof}
By passing to a field extension, we may assume that $\F$ is algebraically closed.
Then $\GL_2(\F)$ is generated by the scalar multiples of the identity 
and $\SL_2(\F)$. Moreover $\alpha I \in \GL_2(\F)$ acts on $V$ and $W$ as scalar
multiplication by $\alpha^d$ where $d$ is the common degree.
Therefore the $\SL_2(\F)$-isomorphism is a $\GL_2(\F)$-isomorphism.
\end{proof}

\section{Symmetric polynomials and \texorpdfstring{$\mathrm{SL}_2(\mathbb{F})$}{SL2(F)}-plethysms}\label{subsec:ElemtsToPolys}

\subsection{Symmetric and antisymmetric polynomials}\label{subsec:symmetric-polys}

An \emph{alphabet} $\x$ is a finite set of distinct variables,
denoted $x_1,\ldots, x_\DN$.
We write $|\x|$ for the size~$\DN\in\N_0$ of the alphabet.
Let $\F[\x]$ denote the polynomial algebra $\F[x_1,\ldots, x_\DN]$.
The symmetric group $\SymG_\DN$ acts on $\F[\x]$ by permuting
the variables: thus $\sigma \cdot P = P(x_{\sigma(1)}, \ldots, x_{\sigma(\DN)})$.
The subalgebra
\begin{equation}
\label{eq:Lambda}
\Lambda[\x] = \{P \in \F[\x] : \text{$\sigma \cdot P = P$ for all $\sigma \in \SymG_\DN$} \}
\end{equation}
of $\F[\x]$ is the \emph{algebra of symmetric polynomials in $\x$}.
%It is graded by total degree: $\Lambda[\x] = \bigoplus_i \Lambda_i[\x]$.
%We define
%$\Lambda_{\le d}[\x] = \bigoplus_{i\le d}\Lambda_i[\x]$. If $d=0$ or $|\x|=0$, then $\Lambda_{\leq d}[\x]\cong\F$.

Given a partition $\lambda$ of length $\ell(\lambda)\le D$, we define
\[ a_{\lambda+\rho}(\x) = \det(x_j^{\lambda_i+\DN-i})_{i,j} \]
where $\rho = (\DN-1,\ldots,2,1)$ and
the sum of partitions is taken entry-wise.
Since permuting the rows in a $\DN \times \DN$ matrix by a permutation $\sigma$ changes
its determinant by $\sgn(\sigma)$, each polynomial $a_{\lambda+\rho}(\x)$
is antisymmetric in $\x$; that is, $\sigma \cdot a_{\lambda + \rho}(\x) = \sgn(\sigma) a_{\lambda + \rho}(\x)$ and $a_{\lambda+\rho}(\ldots, x, x, \ldots) = 0$.
For example, 
\[ a_\rho(\x) = \det(x_j^{\DN-i})_{i,j} = \prod_{1 \le i < j \le \DN} (x_i-x_j) \]
is the Vandermonde determinant.
More generally, a result going back to Cauchy~\cite{Cauchy1815} is that the following subspace 
\begin{equation}
\label{eq:A[x]}
A[\x] = \{a_\rho(\x)P(\x) \in \F[\x] : P\in\Lambda[\x] \}
\end{equation}
is the \emph{space of antisymmetric polynomials in $\x$}.
% More generally we have the following classical result, which
% dates back to Cauchy \cite{Cauchy1815}.
% 
% \begin{Lemma}\label{lem:antisymmetricVandermonde}
% A polynomial $P\in\F[\x]$ is antisymmetric in $\x$ if and only if it is a product of a symmetric polynomial $Q\in\Lambda[\x]$ and $a_\rho(\x)$.
% \end{Lemma}
% 
% \begin{proof}
% See \cite[page 40]{Macdonald}.
% \end{proof}
 
%    \begin{proof}
%        % Recall that by the Vandermonde determinant formula, $a_\rho(\x) = \prod_{i<j}(x_i-x_j)$.
%        If $f$ is antisymmetric, then for each $i < j$ the transposition $\tau = (i, j)\in \SymG_{|\x|}$ acts by $\tau f = - f$. Hence $f$ vanishes when $x_i = x_j$, which implies that $f$ is divisible by $(x_i - x_j)$, and consequently $f / a_\rho$ is a polynomial. Furthermore, since $a_\rho$ is antisymmetric, for each transposition $\tau$ we have 
%        \[
%        \tau \frac{f}{a_\rho} = \frac{\tau f}{\tau a_\rho}
%        = \frac{- f}{ - a_\rho} = \frac{f}{a_\rho},
%        \]
%        showing that $f / a_\rho$ is a symmetric polynomial.
%    \end{proof}

\begin{Definition}\label{defn:Schur}
The \emph{Schur polynomial} in $\x$ labelled by the partition $\lambda$ 
is 
\[s_\lambda(\x) = a_{\lambda+\rho}(\x)/a_\rho(\x). \]
\end{Definition}

Since $a_{\lambda + \rho}(\x)$ is antisymmetric then %by Lemma~\ref{lem:antisymmetricVandermonde},
$s_\lambda(\x)$ is a symmetric polynomial. 
The Schur polynomials for partitions $\lambda$ of $d$ having at most $\DN$ parts (where $\DN=|\x|$) are
a basis for $\Lambda_{\le d}[\x]$.
For further background on symmetric functions and symmetric polynomials
we refer the reader to \cite[Ch.~7]{StanleyEC2} or~\cite{Macdonald};
in particular for the equivalence of Definition~\ref{defn:Schur} with~\eqref{eq:hookCharacter}
when $\lambda$ is a hook partition, see \cite[Theorem~7.15.1]{StanleyEC2}.

\begin{Remark}\label{remark:field}
An important feature of our construction of the algebra $\Lambda[\x]$ 
of symmetric polynomials is that the field $\F$ is arbitrary.
The structure constants of $\Lambda[\x]$ depend both on $\F$ and on the size of $\x$. For instance, if $|\x| = 3$ then
\[
s_{(2,1)}^2(\x) = s_{(2,2,2)}(\x) + 2 s_{(3,2,1)}(\x) + s_{(3,3)}(\x) + s_{(4,1,1)}(\x) + s_{(4,2)}(\x)
\]
for $\F=\C$, but the term $2s_{(3,2,1)}(\x)$ vanishes when $\F$ has characteristic 2. 
If instead $\x$ has size $4$ then $s_{(2,2,1,1)}$ also appears as a summand.
\end{Remark}

\subsection{Polynomial identification of symmetric power plethysm}\label{subsec:construction}
Let $E$ be the 2-dimensional vector space over $\mathbb{F}$ with basis $X, Y$.
In \S\ref{subsec:upperSymmetricPowers}, we defined $\Sym^d \E$ to be the degree $d$ homogeneous component $\F_d[X,Y]$ of the graded polynomial algebra $\F[X,Y]$. 
%, writing the elements of $\Sym^d \E$ as $f(X,Y)$.
Motivated by this construction, we now  present similar \textit{polynomial interpretations} of plethysms $\Sym_N\Sym^d E$ and $\bwedge{N}\Sym^d E$.

Let $\F_d[\XX,\YY]$ be the degree $d$ multi-homogeneous (in each pair $X_i, Y_i$) component of the grading of $\F[\XX,\YY]$, where $|\XX|=|\YY|=N$. 
%in a sense that the total degree in $X_i,Y_i$ is $d$ for $i=1,\ldots,N$. 
Recall that $\SL_2(\F)$ acts on $\F_d[\XX,\YY]$ component-wise via rule \eqref{action-of-SL_2} so $(\Sym^d\E)^{\otimes N}$ can be viewed as $\F_d[\XX,\YY]$.

The symmetric group $\SymG_N$ acts on $\F_d[\XX,\YY]$ by letting $\sigma\in \SymG_N$ send $X_i$ to $X_{\sigma(i)}$ and $Y_i$ to $Y_{\sigma(i)}$ for all $i$. Let $\F_d[\XX, \YY]^{\SymG_N}$ be the invariant subspace of this action. 
The action of $\SL_2(\F)$ commutes with the action of $\SymG_N$, inducing an action on $\F_d[\XX, \YY]^{\SymG_N}$. As an $\SL_2(\F)$-representation, $\F_d[\XX, \YY]^{\SymG_N}$
is therefore isomorphic to $\Sym_N\Sym^d\E$.

For ease of notation in later parts of the paper, we specialize $Y_i$ to 1 for all $i=1,\ldots,N$. Since the polynomials we consider are homogeneous in the total degree of $X_i, Y_i$, this evaluation is invertible, with inverse
$X_i^{\alpha_i}\leftrightarrow X_i^{\alpha_i} Y_i^{d-\alpha_i}$ for all $i$. We thus obtain an isomorphism of vector spaces
\begin{align}\label{eq: dehomogenisation}
\ev_\YY: \F_d[\XX, \YY]^{\SymG_N}\cong\Lambda_{\leq d}[\XX]
\end{align}
via $Y_i\mapsto 1$ for all $i$, where the right-hand side is the space of symmetric polynomials in $\XX$ in which the degree of \emph{each variable} is bounded above by $d$. We endow $\Lambda_{\leq d}[\XX]$ with an $\SL_2(\F)$-action (and even a $\GL_2(\F)$-action) so that it intertwines with the action on $\F_d[\XX, \YY]^{\SymG_N}$ under the above isomorphism $\ev_\YY$. As such, for the rest of the paper we work with the $\SL_2(\F)$-representation $\Lambda_{\leq d}[\XX]$, isomorphic to $\Sym_N\Sym^d \E$ by the construction in this subsection. This action is described explicitly at the end of this section.

%\begin{figure}[h]
%    \centering
%    \[\begin{tikzcd}
%	{\F_d[\XX, \YY]^{\SymG_N}} & {\Sym_N\Sym^d \E} \\
%	{\Lambda_{\leq d}[\XX]}
%	\arrow[tail reversed, from=1-1, to=1-2]
%	\arrow[tail reversed, from=1-1, to=2-1]
%	\arrow[dashed, tail reversed, from=2-1, to=1-2]
%\end{tikzcd}\]
%    \caption{\centering By construction above, the three $\SL_2(\F)$-representations are isomorphic}
%\end{figure}

\begin{Remark}\label{remark:identification-of-wedge}
  Recall from~\eqref{eq:A[x]} that 
  $A[\x] = a_\rho(\x) \Lambda[\x]$. Note that the degree in each variable of $a_\rho(\x)$ is $N-1$. By the above identification, if $N-1\le d$, then the $\SL_2(\F)$-plethysm $\Ext^{N}\Sym^{d}E$ is identified with the vector space
$A_{\leq d-N+1}[\x]$
% $a_{\rho}(\x)\Lambda_{\leq d-N+1}[\x]$
of antisymmetric polynomials of degree at most $d-N+1$ in each of their $|\x|=N$ variables.
\end{Remark}

\begin{Remark}
	The process described above passing from representations to $\Lambda[\x]$ is different to the process of taking Weyl characters. In particular, note that we are assigning a symmetric polynomial to \emph{an element} of a representation, and not to a representation itself.
\end{Remark}

To illustrate the usefulness of this identification, we present a brief proof of the modular Wronskian isomorphism; for comparison, the original proof in \cite{McDowellWildon} 
takes four pages.

\begin{Proposition}[Wronskian isomorphism \cite{McDowellWildon, Grinberg}]\label{prop:Wronskian}
	Let $N, d \in \N_0$ be such that $N\le d+1$. There is an isomorphism of $\SL_2(\F)$-plethysms
	\[
	\Sym_N\Sym^{d}\E\cong \Ext^N\Sym^{d+N-1}\E. 
	\]
\end{Proposition}
\begin{proof}
	Let $\x$ be an alphabet of size $N$. Consider the map
	\begin{align*}
		\zeta : \Lambda_{\le d}[\x] &\longrightarrow
		A_{\leq d+N-1}[\x]\\
		P(\x) &\longmapsto a_\rho(\x) P(\x).
	\end{align*}
	Since polynomial algebras are integral domains, the map $\zeta$ is injective.
	Since $a_\rho$ is, by definition, the Vandermonde determinant, each
	$g\in\SL_2(\F)$ acts on it trivially, therefore 
	\[
	g\cdot(a_\rho P) = (g\cdot a_\rho)(g\cdot P) = a_\rho (g\cdot P),
	\]
	which shows that $\zeta$ is $\SL_2(\F)$-equivariant. Finally, by~\eqref{eq:dim-plet},
	\begin{align*}
		\dim\Sym_N\Sym^d \E=\binom{d+N}{N}=\dim\Ext^N\Sym^{d+N-1}E.
	\end{align*}
	An injective $\SL_2(\F)$-equivariant linear map between vector spaces of equal dimensions is an isomorphism of $\SL_2(\F)$-representations. Under the polynomial identification above, $\zeta$~gives the desired isomorphism \eqref{eq:Wronskian}.
\end{proof}

\begin{Remark}
	It is not very hard to see that the isomorphism
	 $\zeta$ coincides with the isomorphism defined (in a different way) in Theorem~1.4 of \cite{McDowellWildon}. We omit further details as this fact is not relevant to the remainder
	 of this paper.
\end{Remark}

\begin{Example}
	In our polynomial interpretation, one basis of $\Sym_N\Sym^d\E$ corresponds to Schur polynomials $s_\lambda(\x)$ with $|\x|=N$ and $\lambda_1\le d$. The map $\zeta$ sends a basis element $s_\lambda(\x)$ to $a_{\lambda+\rho}(\x)$. In particular, the antisymmetric polynomials $a_{\lambda+\rho}(\x)$ with $\lambda_1 \le d$ correspond to the canonical basis elements of $\Ext^N\Sym^{d+N-1}\E$.
\end{Example}

We finish this section by describing explicitly the action of $\GL_2(\F)$ on the space $\Lambda_{\le d}[\x]$. 
Consider a matrix
  \[
    g = \left(\begin{matrix}\alpha&\beta\\\gamma&\delta\end{matrix}\right) \in \GL_2(\F).
  \]
  Recall from~\eqref{action-of-SL_2} that $g$ acts on $\Sym^d\E = \F_d[X,Y]$ by $X\mapsto \alpha X+\gamma Y$ and $Y\mapsto \beta X + \delta Y$. Since the action on tensor products is diagonal, we thus have an action on $\F_d[\XX, \YY]^{\SymG_N}$ by
  \[
    g.P(\XX,\YY) = P(\alpha\XX+\gamma\YY, \beta\XX+\delta\YY),
  \]
  where $\alpha\XX+\gamma\YY$ is a shorthand for the alphabet $(\alpha X_1+\gamma Y_1, \ldots, \alpha X_N+\gamma Y_N)$, and $\beta \XX+\delta\YY$ is defined similarly.
  Since the polynomials are of homogeneous degree $d$ in each pair of variables $X_i, Y_i$, we can equivalently write
  \[
    g.P(\XX,\YY) = (\beta\XX+\delta\YY)^d~P\Big(\frac{\alpha\XX+\gamma\YY}{\beta\XX+\delta\YY},1,\ldots,1\Big),
  \]
  where
  \[
    \frac{\alpha\XX+\gamma\YY}{\beta\XX+\delta\YY} = \Big(
    \frac{\alpha X_1+\gamma Y_1}{\beta X_1+\delta Y_1}, ~\ldots,~
    \frac{\alpha X_N+\gamma Y_N}{\beta X_N+\delta Y_N}\Big)
    \quad\text{and}\quad
    (\beta \XX+\delta\YY)^d = \prod_{i=1}^N (\beta X_i+\delta Y_i)^d
    .
  \]
  After applying the isomorphism $\ev_\YY : \F_d[\XX,\YY]\to\Lambda_{\le d}[\XX]$, we obtain an action of $g$ on this latter space by
  \begin{equation}\label{eq:action of GL2 on sym}
    g.P(\XX) = (\beta \XX+\delta)^d ~P\Big(\frac{\alpha\XX + \gamma}{\beta\XX+\delta}\Big).
  \end{equation}

\subsection{Lagrange interpolation}

This classical result stated below is critical to the proofs 
of Lemma~\ref{lem:psiLands} and Proposition~\ref{prop:left-inverse}.
% 'key' made more precise 'later parts' made more precise

\begin{Proposition}[Lagrange interpolation]\label{prop:lagrange}
	Let $P\in \F[z]$ be a polynomial of degree $d$. Given $D>d$ and $D$ distinct 
	nodes $x_1,\ldots,x_D\in \F$, we have
	\begin{align*}
		P(z)=\sum_{j=1}^D\prod_{i\neq j}\frac{z-x_i}{x_j-x_i}P(x_j).
	\end{align*}
\end{Proposition}
\begin{proof}
	Note that
	\(
	P(z)-\sum_{j=1}^D\prod_{i\neq j}\frac{z-x_i}{x_j-x_i}P(x_j)
	\) 
	is a polynomial of degree at most $D-1$ having at least $D$ roots (namely 
	the nodes $x_1, \ldots, x_D$), so it is identically 0.
\end{proof}

Later, we shall use Lagrange interpolation
over a field with some extra formal variables.
Consider the natural action of the group algebra $\F \SymG_{|\x|+|\y|}$ on $\F(\x,\y)$ by place permutation of the transcendental elements $\x, \y$. % and $(k\sigma + \tau)f = k\sigma f + \tau f$.
For $1 \leq j \leq |\y|$, let 
\begin{align}\label{eq:t_j}
	t_j = \sum_{i=1}^{|\x|} (x_i, y_j)\in \F \SymG_{|\x|+|\y|}
\end{align}
be the sum of all transpositions swapping an element of $\x$ with a fixed element $y_j$ of $\y$.
These group algebra elements permit a concise form for Lagrange interpolation.

\begin{Example}
  Let $P\in\F(y_1,\ldots,y_{j-1},y_{j+1},\ldots,y_M)[z]$ be a polynomial of degree at most $N$. Let $\x = \{x_1,\ldots,x_N\}$ be an alphabet. Then the following identity holds in $\F(\x,\y)[z]$
%  identity in Proposition~\ref{prop:lagrange} can be written as
	\begin{equation}
		\label{eq:Lagrange}
		P(z) = (1+t_j)\Bigg(\prod_{i=1}^{N}\frac{z-x_i}{y_j-x_i}P(y_j)\Bigg).
        \end{equation}
  The formula is obtained from Lagrange interpolation with $D=N+1$ nodes $x_1, \ldots, x_N, y_j$. Note that the assumption on the base field of $P$ ensures that the group algebra element $(x_i, y_j)$ transforms $P(y_j)$ into $P(x_i)$, rather than transforming the coefficients of $P$ as well.
\end{Example}

\subsection{Polynomial identification of hook Weyl module plethysm}\label{subsec:poly_weyl}
As a final preliminary, in
this section we construct a polynomial identification of $\Delta^{(M+1, 1^{N-1})}\Sym^d \E$ from Theorem~\ref{thm:hook}, using the setting of \S\ref{subsec:construction}. 
Recall the map $\delta_M$ defined in \eqref{eq:muM}. Directly from the definition of $v_{c_1}\wedge\dots\wedge v_{c_R}$ in \S\ref{subsec:exteriorLowerSymmetricPowers}, we have
% In this section, the symmetric group $\SymG_{N+1}$ permutes the first $N+1$ tensor factors of the element of of $\{c_0, c_1, b_1,\ldots, b_{N-1}\}$. Consider its subgroup $S$ isomorphic to $\SymG_N$ permuting the subset $\{c_0,b_1,\ldots, b_{N-1}\}$. Then the transpositions $T=\{(c_1, c_0), (c_1, b_1), \ldots, (c_1, b_{N-1})\}$ form a set of coset representatives of $\SymG_{N+1}/S$. Hence,
\[\begin{split}
	v_{c_0} \wedge v_{b_1} &\wedge \cdots \wedge v_{b_{N-1}} 
	\wedge v_{c_1} \otimes v_{c_2} \otimes \cdots \otimes v_{c_M} = \\&\sum_{\tau}\text{sgn}(\tau)\tau\cdot(v_{c_0} \wedge v_{b_1} \wedge \cdots \wedge v_{b_{N-1}}
	\otimes v_{c_1}\otimes v_{c_2}\otimes \cdots \otimes v_{c_M}),   \end{split}
\]
where the sum is over the identity permutation and all transpositions $\tau\in\SymG_{N+1}$ that act by
swapping the $(N+1)$st tensor factor with one of the first $N$ tensor factors.
Since $\delta_M$ is a linear map, for each $v\in\bwedge{N}V\otimes \Sym_M V$ we have
\[
\delta_M(v)=\sum_{\tau}\text{sgn}(\tau)\tau\cdot v.
\]
Now under the polynomial identification of symmetric powers from \S\ref{subsec:construction}, 
taking $V=\Sym^d\E $
the map $\delta_M$ having kernel $\Delta^{(M+1,1^{N-1})}\Sym^d \E$ may be re-expressed as
\begin{align}\label{eq:poly-deltaM}
	\delta_M: A_{\leq d}[\x]\otimes\Lambda_{\leq d}[\y]&\longrightarrow A_{\leq d}[\x, y_1]\otimes\Lambda_{\leq d}[y_2,\ldots,y_M]\notag\\
	P(\x, \y)&\longmapsto (1-t_1)P(\x, \y),
\end{align}
with $|\x|=N, |\y|=M$, and $t_1 \in \F\SymG_{N+M}$ defined in \eqref{eq:t_j}.

\begin{Remark}The variables $\y$ should not be confused with the variables~$\YY$ from previous sections; in particular, the polynomials $P(\x, \y)$ are not necessarily homogeneous in the degree of $x_i$ and $y_i$ and even the number of $x$'s and $y$'s need not be equal. 
\end{Remark}

Hence, by Lemma \ref{lemma:DeltaKernel}, we identify $\Delta^{(M+1, 1^{N-1})}\Sym^d \E$ with the kernel of $\delta_M$, considered as the polynomial map~\eqref{eq:poly-deltaM}.

\section{Definitions of the isomorphisms and their basic properties}\label{sec:def-of-the-morph}

In this section, we define two polynomial evaluation maps, $\psi$ %, $\kappa$,
and $\phi$, that we shall later show they induce the isomorphisms in Theorems \ref{thm:hook} %, \ref{thm:T&L1},
and \ref{thm:GTL}. We prove that these maps are well-defined, $\SL_2(\F)$-equivariant, and admit %respective 
left inverses. The descriptions of the maps are very similar; the main differences lie in their respective domains and codomains. As a result, in some cases, we can verify the above properties simultaneously.

\begin{Definition}\label{def:psi}
	Let $M,d\in\N_0$ and $N\in\N$ such that $N\le d+1$. Let $\x,\y,\z$ be alphabets with $|\x|=N$ and $|\y|=|\z|=M$. Define $\psi$ to be the evaluation map
	\begin{align*}
		\psi : 
		\Lambda_{\le N-1}[\z]\otimes\Lambda_{\le d-N+1}[\x,\y]
		&\longrightarrow
		\Lambda_{\le d-N+1}[\x]\otimes\Lambda_{\le d}[\y]\\
		P(\x,\y,\z) &\longmapsto P(\x,\y,\y).
	\end{align*}
\end{Definition}

\begin{Lemma}\label{lem:psiLands}
	Let 
	$\zeta$ be the Wronskian isomorphism from Proposition
	\ref{prop:Wronskian}. 
	With the identifications from \S\ref{subsec:construction} and \S\ref{subsec:poly_weyl}, the image of $$\Sym_{M}\Sym^{N-1}\E\otimes\Sym_{M+N}\Sym^{d-N+1}\E$$ under the linear map $(\zeta\otimes1)\circ\psi$ 
	is a subspace of $\Delta^{(M+1,1^{N-1})}\Sym^{d}\E$.
\end{Lemma}
\begin{proof}
	Recall from \S\ref{subsec:ElemtsToPolys} that an element of $\Sym_{M}\Sym^{N-1}\E\otimes\Sym_{M+N}\Sym^{d-N+1}\E$ is identified with a symmetric polynomial $P(\x,\y,\z) = \sum_i Q_i(\z)\otimes R_i(\x,\y)\in\Lambda_{\le N-1}[\z]\otimes\Lambda_{\le d-N+1}[\x,\y]$. 
	By linearity, we may assume without loss of generality, 
	that this symmetric polynomial is a pure tensor $Q(\z)\otimes R(\x,\y)$.
	
	First, we show that $\psi$ is well defined. After evaluating $\z$ to $\y$ by the specialization
	 $z_i\mapsto y_i$, the image of $Q\otimes R$ is symmetric in $\y$ and symmetric in $\x$. Counting degrees, we have
	\[
	\psi(Q\otimes R) = Q(\y)R(\x,\y) \in \Lambda_{\le d-N+1}[\x]\otimes\Lambda_{\le d}[\y].
	\]
	Next, apply the Wronskian isomorphism to the first tensor factor to get
	\[
	\bigl((\zeta\otimes 1)\circ\psi\bigr)(Q\otimes R)\in A_{\le d}[\x]\otimes\Lambda_{\le d}[\y],
	\]
	so by \eqref{eq:poly-deltaM}, it remains to show that the right-hand side is in the kernel of $\delta_M$. That is,
	\[
	\bigl(\delta_M\circ(\zeta\otimes1)\circ\psi\bigr)(Q\otimes R)=0.
	\]
	Indeed, $R(\x,\y)$ is symmetric in $\x$ and $\y$, so 
	% $t_1(R(\x,\y))=N\cdot R(\x,\y)$, and hence
	\begin{align*}
		\bigl(\delta_M\circ(\zeta&{}\otimes1)\circ\psi\bigr)(Q\otimes R)
		\\&=(1-t_1)\bigl(a_\rho(\x)Q(\y)R(\x,\y)\bigr)\\
		&=R(\x,\y)\cdot(1-t_1)\bigl(a_\rho(\x)Q(\y)\bigr)\\
		&=R(\x,\y)\cdot a_\rho(\x)\cdot\left(Q(\y)-\sum_{i=1}^N\prod_{j\neq i}\frac{y_1-x_j}{x_i-x_j}Q\bigl(x_i, y_2,\ldots, y_M)\right)\\ &=0,
	\end{align*}
	where the final step follows from Lagrange interpolation (Proposition~\ref{prop:lagrange}) of the polynomial $\tilde{Q}(z)=Q(z,y_2,\ldots,y_M)\in\F(y_2,\ldots,y_M)[z]$ of degree $\deg\tilde{Q}=\deg_z Q\le N-1$ on the $N$ nodes $x_1,\ldots, x_N$, evaluated at $z=y_1$.
\end{proof}
\begin{Example}\label{eg: hook for M=1}
	One can derive combinatorial formulas for the map $(\zeta\otimes1)\circ\psi$ on classical bases. Recall that $|\y|=|\z|=M$ and $|\x|=N$. For simplicity in this example, we set $M=1$. Then \[
	\psi : \Lambda_{\le N-1}[z]\otimes\Lambda_{\le d-N+1}[\x,y]
	\longrightarrow
	\Lambda_{\le d-N+1}[\x]\otimes\Lambda_{\le d}[y].
	\]
	Under the identifications from \S\ref{subsec:ElemtsToPolys}, we have
	\[
	(\zeta\otimes1)\circ\psi : \Sym^{N-1}\E\otimes\Sym_{N+1}\Sym^{d-N+1}\E \longrightarrow \Delta^{(2,1^{N-1})}\Sym^{d}\E.
	\]
	Consider a basis element $z^n s_{\lambda}(\x, y)$ of the domain for some $n \le N-1$ and $\lambda_1 \le d-N+1$, where $s_\lambda$ denotes the Schur polynomial from Definition \ref{defn:Schur}.
	Then $\psi$ maps this element to $y^n s_{\lambda}(\x, y)$.
	Using skew Schur functions (see \cite[I, \S5]{Macdonald} or \cite[\S7.10]{StanleyEC2})
	and the plethystic addition formula
	$s_\lambda(\tilde{\x}, \tilde{\y}) = \sum_{\mu \subseteq \lambda}
	s_\mu(\tilde{\x})s_{\lambda/\mu}(\tilde{\y})$ (see \cite[I, (8.8)]{Macdonald}),
	$y^n s_{\lambda}(\x, y)$ can be decomposed into a sum of pure tensors, giving
	\[
	y^n s_{\lambda}(\x, y) = y^n\sum_{\mu\subseteq\lambda}s_\mu(\x)s_{\lambda/\mu}(y).
	\]
	Finally, $\zeta\otimes1$ sends the right-hand side 
	to $y^n\sum_{\mu\subseteq\lambda}a_{\mu+\rho}(\x)s_{\lambda/\mu}(y)$. It is notable that the plethystic addition formula 
	corresponds to this `separation of variables' in our symmetric functions model.
\end{Example}
\begin{Remark}
	The isomorphism $$\varphi : \Sym^{N-1}\E \otimes \Ext^{N+1} \Sym^{d+1}\E \to \Delta^{(2,1^{N-1})} \Sym^d\E$$ constructed in \cite[(1.8)]{MW24} is precisely $(\zeta\otimes 1)\circ\psi\circ(1\otimes\zeta^{-1})$. Explicitly,
	\[
	\varphi :
	z^n a_{\lambda+\rho}(\x,y) \mapsto
	y^n\sum_{\mu\subseteq\lambda}a_{\mu+\rho}(\x)s_{\lambda/\mu}(y)
	\]
	on basis elements.
	It is far from obvious that this description of $\varphi$ coincides with the map defined in~\cite{MW24}, and the only proof that the authors have occupies several pages. Since we believe
	the description of $\varphi$ given in this paper is the most useful for further work, we shall
	not give more details of the proof of this remark here.
\end{Remark}

\begin{Definition}\label{def:phi}
	Let $M, N, d \in \N_0$, and let $\x,\y,\z$ be alphabets with $|\x|=N$ and $|\y|=|\z|=M$. Define $\phi$ to be the map
	\begin{align*}
		\phi : 
		\Lambda_{\le N}[\z]\otimes\Lambda_{\le d}[\x,\y]
		&\longrightarrow
		\Lambda_{\le d}[\x]\otimes\Lambda_{\le d+N}[\y]\\
		P(\x,\y,\z) &\longmapsto P(\x,\y,\y).
	\end{align*}    
\end{Definition}
\begin{Lemma}\label{lem:phiLands}
	With the identification from \S\ref{subsec:construction}, the map $\phi$ induces a linear map
	\[
	\phi : \Sym_M\Sym^N\E\otimes\Sym_{N+M}\Sym^d\E \longrightarrow \Sym_N\Sym^{d}\E\otimes\Sym_M\Sym^{d+N}\E
	\]
	of vector spaces.
\end{Lemma}
\begin{proof}
	This follows directly from the definition of $\phi$ and the identifications in \S\ref{subsec:ElemtsToPolys}.
\end{proof}

To show the $\SL_2(\F)$-equivariance and construct left inverses for $\psi$ and $\phi$, we consider a more general map. 

\begin{Definition}\label{def:pi}
	Let $M,N,a,b\in\N_0$, and let $\x,\y,\z$ be alphabets with $|\x|=N$ and $|\y|=|\z|=M$. Define
	\begin{align*}
		\pi:\Lambda_{\leq\deltaD}[\z]\otimes\Lambda_{\leq b}[\x,\y] &\longrightarrow \Lambda_{\leq b}[\x] \otimes\Lambda_{\leq b+\deltaD}[\y]\\ 
		P(\x,\y,\z) &\longmapsto P(\x,\y,\y).
	\end{align*}
\end{Definition}

\begin{Remark}
	When $\deltaD=N-1$ and $b=d-N+1$, the map $\pi$ coincides with $\psi$.
	When $\deltaD=N$ and $b=d$, the map $\pi$ coincides with $\phi$.
\end{Remark}

\begin{Proposition}\label{prop:SL2invariance}
	The map $\pi$ is $\GL_2(\F)$-equivariant.
\end{Proposition}
The following proof was suggested to us by Darij Grinberg as a simplification of our original approach.
\begin{proof}
  %We are now ready to show $\GL_2(\F)$-equivariance of the map $\pi$.
  Recall from~\eqref{eq:action of GL2 on sym} that a matrix $g=\left(\begin{smallmatrix}\alpha&\beta\\\gamma&\delta\end{smallmatrix}\right)\in\GL_2(\F)$ acts on $\Lambda_{\le\deltaD}[\z]\otimes\Lambda_{\le b}[\x,\y]$ by
  \[
    g.P(\x,\y,\z) = (\beta\x+\delta)^b (\beta\y+\delta)^b (\beta\z+\delta)^\deltaD ~P\Big(\frac{\alpha\x + \gamma}{\beta\x+\delta},~\frac{\alpha\y + \gamma}{\beta\y+\delta},~\frac{\alpha\z + \gamma}{\beta\z+\delta}\Big).
  \]
  After applying $\pi$, we obtain
  \[
    (\beta\x+\delta)^b (\beta\y+\delta)^{b+\deltaD} ~P\Big(\frac{\alpha\x + \gamma}{\beta\x+\delta},~\frac{\alpha\y + \gamma}{\beta\y+\delta},~\frac{\alpha\y + \gamma}{\beta\y+\delta}\Big),
  \]
  which coincides with the action of $g$ on $P(\x,\y,\y)$.%$\Lambda_{\le b}[\x]\otimes\Lambda_{\le b+\beta}[\y]$.
\end{proof}

To conclude this section, we construct a left inverse $\tilde\pi$ of $\pi$.
Inspired by~\eqref{eq:Lagrange}, for each $1\leq j\leq |\y|$ we define an operator
\begin{align}\label{eq:Lj}
	\mathcal{L}_j: \F(\x,\y)[\z] & \longrightarrow \F(\x, \y)[\z]\\
	P&\longmapsto (1+t_j)\Big(P\cdot\prod_{i=1}^{|\x|}\frac{z_{j}-x_i}{y_j-x_i}\Big),\notag
\end{align}
where $t_j$ is given by \eqref{eq:t_j}. Informally, these operators will be used to `recover' the variable~$z_j$ after it disappears under the evaluation $z_j\mapsto y_j$; this is made precise in Example~\ref{ex:recovering-z} following the proposition below.
% Note that if a polynomial map $\pi:\F[\x,\y,\z]\to\F[\x,\y]$ has a rational inverse $\F(\x,\y)\to\F(\x,\y)[\z]$, then the inverse is a polynomial map when restricted to $\mathrm{im}(\pi)\subseteq\F[\x,\y]$.

% The example above exhibits the intuition behind the definition of a left inverse of $\pi$ in the proposition below. Crucially, the symmetry of $s_\lambda$ allowed us to recover the variable $z$ precisely form the $y$'s that appeared via evaluation $\psi$.

\begin{Proposition}\label{prop:left-inverse}
  Let $\deltaD\leq N$. For $f$ in the image of $\pi$, define $\tilde\pi : f \mapsto \mathcal{L}_1\mathcal{L}_2\dots \mathcal{L}_{M}(f)$. Then $\tilde{\pi}\circ\pi$ is the identity.
\end{Proposition}
\begin{proof}
	% First note, that the image of $\mathcal{L}_1\mathcal{L}_2\dots \mathcal{L}_{M}\vert_{\F[\x,\y]}$ lies in $\F[\x,\y,\z]$. 
	% The well-definedness of $\tilde\pi$ will follow from its construction.
	
	By linearity, it suffices to show that $\tilde\pi\circ\pi$ acts as the identity on pure tensors $Q(\z)\otimes R(\x, \y)$ in $\Lambda_{\leq \deltaD}[\z]\otimes\Lambda_{\leq b}[\x,\y]$. The definitions of $\pi$ and $\tilde{\pi}$ give
	\begin{align*}
		(\tilde\pi\circ\pi)\big( Q(\z)R(\x,\y) \big) &=
		\tilde\pi\big( Q(\y)R(\x,\y) \big) \\
		&= \mathcal{L}_1\mathcal{L}_2\dots \mathcal{L}_M\big( Q(\y)R(\x,\y) \big).
	\end{align*}    
	Since $R$ is symmetric in $\x\,\cup\,\y$, we have $(x_i,y_j) R (\x,\y)= R(\x, \y)$. Therefore, 
	\[ \mathcal{L}_j(Q(\y)R(\x,\y)) =\mathcal{L}_j(Q(\y))\cdot R(\x,\y).\]
	 We claim by induction that
	\begin{align*}
		\mathcal{L}_j\dots \mathcal{L}_M(Q(\y)R(\x,\y))=Q(y_1,\ldots,y_{j-1},z_j,\ldots,z_M) R(\x,\y).
	\end{align*}
	The base case $j=M$ follows directly from Lagrange interpolation (Proposition~\ref{prop:lagrange}), and so does the inductive step:
	\begin{align*}
		\mathcal{L}_j\big(Q(y_1,\ldots,y_j,&\,z_{j+1},...,z_M)R(\x,\y)\big) \\
		% &=
		% (1+t_j)\Big(g(\y_j,z_{j+1},...,z_m)h(\x,\y) \cdot\prod\nolimits_{i=1}^{|\x|}\frac{z_j-x_i}{y_j-x_i} \Big)\\
		&= (1+t_j)\Big(Q(y_1,\ldots,y_j,z_{j+1},...,z_M) \cdot\prod\nolimits_{i=1}^{N}\frac{z_j-x_i}{y_j-x_i} \Big)\cdot  R(\x,\y)\\
		&\stackrel{\eqref{eq:Lagrange}}{=} Q(y_1,\ldots,y_{j-1},z_j,...,z_M) R(\x,\y)\,.
	\end{align*}
	In the last step we apply Lagrange interpolation at the $N+1$ nodes $x_1,\ldots, x_N,y_j$ for the polynomial $\tilde{Q}(Z)=Q(y_1,\ldots,y_{j-1},Z,z_{j+1},\ldots, z_M)\in\F(y_1,\ldots,y_{j-1},z_{j+1},\ldots, z_M)[Z]$ of degree 
	\[ \deg \tilde{Q}=\deg_ZQ=\deltaD<N+1. \]
		We conclude by induction that 
	\begin{align*}
		\mathcal{L}_1\mathcal{L}_2\dots \mathcal{L}_{M}\big( Q(\y)R(\x,\y) \big) = Q(\z)R(\x,\y),
	\end{align*}
	as desired.
\end{proof}

\begin{Example}\label{ex:recovering-z}
	Let $N=1$ and consider a pure tensor $z\cdot P(x,y)$ of $\Lambda_{\leq1}[z]\otimes\Lambda_{\leq d-1}[x,y]$. Then $\pi(z\cdot P(x,y))=y\cdot P(x,y)$, and $\tilde{\pi}=\mathcal{L}_1$. We can check directly that:
	\begin{align*}
		\mathcal{L}_1&(y\cdot P(\x,y))=      
		(1+(x, y))\left(y\cdot P(x,y)\cdot\frac{z-x}{y-x}\right)\\
		&=y\cdot P(x,y)\cdot\frac{z-x}{y-x}+x\cdot P(y,x)\cdot\frac{z-y}{x-y}\\
		&=\left(\frac{yz-yx}{y-x}+\frac{xz-xy}{x-y}\right)\cdot P(x,y)\\
		&=z\cdot P(x,y),
	\end{align*}
	so indeed $\mathcal{L}_1$ recovered $z$ after the evaluation $z\mapsto y$, and $\tilde{\pi}=\mathcal{L}_1$ is a left inverse of $\pi$ for $N=M=\deltaD=1, b=d-1$.
\end{Example}

\section{Hook plethysms --- Proof of Theorem \ref{thm:hook}}\label{sec:proof-of-hook}

We prove Theorem \ref{thm:hook} by using the results in the previous section to show
that $(\zeta\otimes1)\circ\psi$ is an isomorphism
from $\Sym_{M}\Sym^{N-1}\E\, \otimes\, \Sym_{M+N}\Sym^{d-N+1}\E$ to
$\Delta^{(M+1,1^{N-1})}\Sym^{d}\E$.

\begin{proof}[Proof of Theorem \ref{thm:hook}]
	Recall from Definition \ref{def:psi} that $\psi$ is a map on symmetric polynomials given by:
	\begin{align*}
		\psi : 
		\Lambda_{\le N-1}[\z]\otimes\Lambda_{\le d-N+1}[\x,\y]
		&\longrightarrow
		\Lambda_{\le d-N+1}[\x]\otimes\Lambda_{\le d}[\y]\\
		P(\x,\y,\z) &\longmapsto P(\x,\y,\y).
	\end{align*}
	By Lemma \ref{lem:psiLands} and Proposition \ref{prop:SL2invariance}, it is an $\SL_2(\F)$-homomorphism with a left inverse $\tilde\psi$ defined in Proposition \ref{prop:left-inverse}. Therefore, the composition
	\[
	(\zeta\otimes1)\,\circ\,\psi : \Sym_{M}\Sym^{N-1}\E\, \otimes\, \Sym_{M+N}\Sym^{d-N+1}\E \longrightarrow \Delta^{(M+1,1^{N-1})}\Sym^{d}\E
	\]
	from Lemma~\ref{lem:psiLands} is an $\SL_2(\F)$-homomorphism with a left inverse $\tilde\psi\circ(\zeta^{-1}\otimes 1)$.
	
	It remains to show that the dimensions of the two representations are equal. Indeed, by~\eqref{eq:dim-plet} and Proposition~\ref{prop:dim-delta}
	\begin{align*}
		\dim& \big(\Sym_{M}\Sym^{N-1}\E\otimes\Sym_{M+N}\Sym^{d-N+1}\E \big)\\
		&=\binom{M+N-1}{M}\binom{M+d+1}{M+N}=\dim\left(\Delta^{(M+1, 1^{N-1})}\Sym^d \E\right).
		\qedhere
	\end{align*} 
\end{proof}

We deduce a corresponding isomorphism of $\GL_2(\F)$-plethysms.

\begin{Corollary}\label{cor:GL2-isom}
	Let $M,d\in\N_0$ and $N\in\N$ such that $N\le d+1$. There is an isomorphism of $\GL_2(\F)$-representations
	\[ 
	{\det}^{\binom{N}{2}}\otimes\Sym_{M}\Sym^{N-1}\E\otimes\Sym_{M+N}\Sym^{d-N+1}\E \cong \Delta^{(M+1,1^{N-1})}\Sym^{d}\E.
	\]
\end{Corollary}

\begin{proof}
By passing to a field extension we may assume that $\F$ is infinite.
The result then follows immediately from Theorem~\ref{thm:hook} by applying Lemma~\ref{lemma:SL2toGL2};
note that
the degree of the polynomial representation on the left-hand side
is $N(N-1) + M(N-1) + (M+N)d - (M+N)(N-1) = (M+N)d$, which agrees with
the right-hand side.
\end{proof}

\section{Trinomial plethysms --- Proof~of~Theorem~\ref{thm:GTL}}

In this section, we gather the results about $\phi$ presented in \S\ref{sec:def-of-the-morph} to prove Theorem \ref{thm:GTL}.
We shall make use of the trinomial revision identity~\eqref{eq:GTL}
\[ \binom{M+N}{M}\binom{M+N+d}{M+N} = \binom{N+d}{N} \binom{M+N+d}{M} \]
stated in the introduction.

\begin{proof}[Proof of Theorem \ref{thm:GTL}]
	Recall from Definition \ref{def:phi} that we have a map $\phi$ of symmetric polynomials:
	\[
	\phi : 
	\Lambda_{\le N}[\z]\otimes\Lambda_{\le d}[\x,\y]
	\longrightarrow
	\Lambda_{\le d}[\x]\otimes\Lambda_{\le d+N}[\y].
	\]
	By Lemma \ref{lem:phiLands} and Proposition \ref{prop:SL2invariance}, it induces an $\SL_2(\F)$-homomorphism 
	\[
	\phi:\Sym_M\Sym^N\E\otimes\Sym_{N+M}\Sym^d\E \rightarrow
	\Sym_N\Sym^{d}\E\otimes\Sym_M\Sym^{d+N}\E
	\]
	with a left inverse $\tilde\phi$ defined in Proposition~\ref{prop:left-inverse}. 
	To conclude the proof, note that by~\eqref{eq:dim-plet} and~\eqref{eq:GTL}
	these representations have equal dimensions. 
\end{proof}

\begin{Corollary}\label{cor:extGTL}
	Applying the Wronskian isomorphism~\eqref{eq:Wronskian}
	and replacing $d$ with $e-M-N+1\geq 0$ we get the equivalent form
	\[ %\begin{equation}
	%\label{eq:GTLafterWronskian}
	\bwedge{M} \Sym^{N+M-1} \E \otimes \bwedge{M+N} \Sym^{e} \E \cong
	\bwedge{N} \Sym^{e-M} \E \otimes \bwedge{M} \Sym^{e} \E \] %\end{equation}
	which is also of note.
\end{Corollary}

\begin{Remark}\label{remark:GTLq}
	%As in the proof of \S\ref{sec:proof-of-hook}, 
	As in Corollary~\ref{cor:GL2-isom}, we deduce a $\GL_2(\F)$-isomorphism
	\[
	\Sym_M\Sym^N\E\otimes\Sym_{N+M}\Sym^d\E \cong
	\Sym_N\Sym^{d}\E\otimes\Sym_M\Sym^{d+N}\E.
	\]
	When $\F=\C$, equating the characters gives the $q$-binomial identity:
	\[
	\qbinom{N+M}{M}\qbinom{d+N+M}{N+M} = \qbinom{d+N}{N}\qbinom{d+N+M}{M},
	\]
	lifting~\eqref{eq:GTL} to $q$-binomial coefficients.
\end{Remark}

\section{Team-and-leader isomorphisms --- Proof~of~Corollary~\ref{cor:T&L}}
\label{sec:TLisomorphisms}
In this section, we deduce the isomorphisms in Corollary \ref{cor:T&L}, and explain why
we refer to them as `team-and-leader' isomorphisms.
For ease of reference we restate Corollary~\ref{cor:T&L} below.

\setcounter{section}{1}
\setcounter{Theorem}{2}
\begin{Corollary}
	Let $K, d \in \N_0$. The following isomorphisms of $\SL_2(\F)$-representations hold:
	\begin{thmlist}
		\item %Don't put in [(i)] by hand as thmlist handles the no italics already
		$
		\Sym^{d+K}\E\otimes\Sym_K\Sym^{d}\E\, \cong
		\Sym^K\E\otimes\Sym_{K+1}\Sym^d\E\,;
		$
		\medskip
		\item 
		$
		\Sym_K\E\otimes\Sym_{K+1}\Sym^{d}\E \cong \Sym^{d}\E\otimes\Sym_K\Sym^{d+1}\E\,;
		$
		\medskip
		\item
		$
		\Sym_{d+K}\E\otimes\Sym_{K}\Sym^{d}\E \cong \Sym_{d}\E\otimes\Sym_K\Sym^{d+1}\E.
		$
		\medskip
	\end{thmlist}
\end{Corollary}
\setcounter{section}{7}

\begin{proof}
	Using that 
	$\Sym_1\Sym^aE\cong \Sym^aE$ and $\Sym_a\Sym^1E\cong\Sym_aE$ as  $\SL_2(\F)$-repre\-sentations,
	part (i) follows from Theorem \ref{thm:GTL} by letting $M = 1$ and $N = K$.
	Similarly, part (ii) follows from Theorem \ref{thm:GTL} by letting $M = K$ and $N = 1$. 

	For part (iii), we use the basic properties of dual representations. Since $E$ is self-dual, we have $(\Sym_a E)^\star\cong \Sym^aE$ (see the discussion after the proof of Prop.~2.11 in \cite{McDowellWildon}), and $(\Sym_a\Sym^b\E)^\star\cong \Sym_b\Sym^a\E$ as $\SL_2(\F)$-plethysms (see the proof of Cor.~1.5 in \cite{McDowellWildon}). 
	
	Taking duals of both sides of (i), and applying either of the isomorphisms from the previous paragraph to each tensor factor accordingly, we obtain
	\[
	\Sym_{d+K}\E\otimes\Sym_d\Sym^K\E
	\cong
	\Sym_K\E\otimes\Sym_d\Sym^{K+1}\E
	.
	\]
	This is equivalent to the identity in part (iii) upon swapping $d$ and $K$.
\end{proof}

The first two isomorphisms decategorify by taking dimensions 
to the first two equalities~in
\begin{equation}\label{eq:TLidentities} 
	(d+K+1)\binom{d+K}{K} = (K+1) \binom{d+K+1}{K+1} = (d+1) \binom{d+K+1}{K}. \end{equation}
The third isomorphism decategorifies to the equality of the left-hand and right-hand sides.
We remark that these identities have appealing combinatorial proofs in the `team-and-leader' model.
For instance, to prove the first equality,
take $d+K+1$ people and, for the left-hand side, suppose that one of them is a dictator
who chooses $K$ of the remaining $d+K$ people
to form their goverment; for the right-hand side, suppose instead
that $K+1$ of the people are democratically elected, and they choose a prime minister
from their number in $K+1$ ways.
For the second equality we invite the reader
to find an interpretation using a presidential model.
The $q$-analogues of these identities are of course given by specializing
Remark~\ref{remark:GTLq} appropriately.

\section{Team hierarchy and Catalan combinatorics}

We conclude with an interesting corollary of Theorem \ref{thm:GTL}, drawing from Catalan combinatorics. In the combinatorial interpretation, we consider teams with a layered leadership structure, generalizing the `team-and-leader' model at the end of the previous
section. 

Set $N_0 = 0$ and let $N_k > \cdots > N_2 > N_1 > N_0 = 0$ be integers. Assume that they are in \emph{generic position}, in the sense that $N_j - N_i \ne N_m - N_p$ for all distinct pairs $j > i \ge 0$ and $m > p \ge 0$.
% Given \emph{generic} positive integers $N_k > \cdots > N_2 > N_1 > 0$, c
Consider the product
\begin{equation}\label{eq:subsetOfSubsetsStart}
	\binom{N_k}{N_{k-1}} \cdots \binom{N_3}{N_2} \binom{N_2}{N_1}.
\end{equation}
% Set $N_0 = 0$. 
As illustrated in Figure \ref{fig:Catalan}, we may successively apply identity \eqref{eq:GTL} to obtain equalities of the form
\begin{equation}\label{eq:subsetOfSubsetsFlip}
	\cdots \binom{N_j - N_{i-1}}{N_m - N_{i-1}}\binom{N_m - N_{i-1}}{N_{p-1} - N_{i-1}} \cdots
	= \cdots \binom{N_j - N_{i-1}}{N_{p-1} - N_{i-1}}\binom{N_j - N_{p-1}}{N_m - N_{p-1}} \cdots
\end{equation}
for some $j > m \ge p > i \ge 1$.
Theorem \ref{thm:GTL} then lifts this identity to an isomorphism of $\SL_2(\F)$-plethysms. 
We call each of the resulting binomial expressions (up to permutation of the factors) an \emph{election process} on a team hierarchy with $k$ layers.

It turns out that the number of election processes in a team hierarchy with $k$ layers is given by the Catalan number~$C_{k-1}$. We show this in Lemma~\ref{lem:catalan-bijection} below, by constructing an explicit bijection between the election processes and rooted binary trees with $k$ leaves.

\begin{figure}[h]
	\centering
	\begin{adjustbox}{scale=0.8}
		\begin{tikzpicture}[rotate=25]
			\def\radius{2cm}
			
			\node (E0) at ({90 + 0*72}:\radius) {
				\begin{tikzpicture}[x=.5em, y=-.5em, rotate=-45, thick]
					\draw (0,0) -- (3,0);
					\draw (0,0) -- (0,3);
					\draw (0,1) -- (2,1);
					\draw (0,2) -- (1,2);
				\end{tikzpicture}
			};
			\node (E1) at ({90 + 1*72}:\radius) {
				\begin{tikzpicture}[x=.5em, y=-.5em, rotate=-45, thick]
					\draw (0,0) -- (3,0);
					\draw (0,0) -- (0,3);
					\draw (2,0) -- (2,1);
					\draw (0,2) -- (1,2);
				\end{tikzpicture}
			};
			\node (E2) at ({90 + 2*72}:\radius) {
				\begin{tikzpicture}[x=.5em, y=-.5em, rotate=-45, thick]
					\draw (0,0) -- (3,0);
					\draw (0,0) -- (0,3);
					\draw (2,0) -- (2,1);
					\draw (1,0) -- (1,2);
				\end{tikzpicture}
			};
			\node (E3) at ({90 + 3*72}:\radius) {
				\begin{tikzpicture}[x=.5em, y=-.5em, rotate=-45, thick]
					\draw (0,0) -- (3,0);
					\draw (0,0) -- (0,3);
					\draw (1,1) -- (2,1);
					\draw (1,0) -- (1,2);
				\end{tikzpicture}
			};
			\node (E4) at ({90 + 4*72}:\radius) {
				\begin{tikzpicture}[x=.5em, y=-.5em, rotate=-45, thick]
					\draw (0,0) -- (3,0);
					\draw (0,0) -- (0,3);
					\draw (0,1) -- (2,1);
					\draw (1,1) -- (1,2);
				\end{tikzpicture}
			};
			
			% Draw pentagon
			\draw[thick, blue, -stealth] (E0) -- (E1);
			\draw[thick, blue, -stealth] (E1) -- (E2);
			\draw[thick, blue, -stealth] (E0) -- (E4);
			\draw[thick, blue, -stealth] (E4) -- (E3);
			\draw[thick, blue, -stealth] (E3) -- (E2);
		\end{tikzpicture}
	\end{adjustbox}
	\hspace{3em}
	\begin{adjustbox}{scale=0.8}
		\begin{tikzpicture}[rotate=25]
			\def\radius{2cm}
			\node (E0) at ({90 + 0*72}:\radius) {
				$\binom{d}{c}\binom{c}{b}\binom{b}{a}$
			};
			\node (E1) at ({90 + 1*72}:\radius) {
				$\binom{d}{b}\binom{d-b}{c-b}\binom{b}{a}$
			};
			\node (E2) at ({90 + 2*72}:\radius) {
				$\binom{d}{a}\binom{d-b}{c-b}\binom{d-a}{b-a}$
			};
			\node (E3) at ({90 + 3*72}:\radius) {
				$\binom{d}{a}\binom{d-a}{c-a}\binom{c-a}{b-a}$
			};
			\node (E4) at ({90 + 4*72}:\radius) {
				$\binom{d}{c}\binom{c}{a}\binom{c-a}{b-a}$
			};
			
			% Draw pentagon
			\draw[thick, red, -stealth] (E0) -- (E1);
			\draw[thick, red, -stealth] (E1) -- (E2);
			\draw[thick, red, -stealth] (E0) -- (E4);
			\draw[thick, red, -stealth] (E4) -- (E3);
			\draw[thick, red, -stealth] (E3) -- (E2);
		\end{tikzpicture}
	\end{adjustbox}
	\caption{On the left, the five rooted binary trees with 4 leaves. An arrow is a right tree rotation. On the right, the five election processes on a team hierarchy with 4 layers of sizes $d > c > b > a$. An arrow is an application of \eqref{eq:subsetOfSubsetsFlip}.}
	\label{fig:Catalan}
\end{figure}
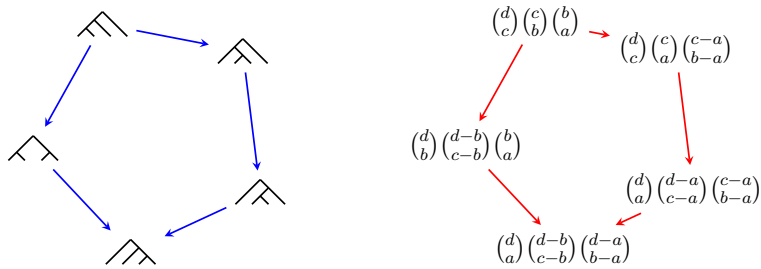

We enumerate the leaves in a rooted binary tree by $1, 2, \ldots, k$ from left to right.
We then label each vertex
by an interval $[i,j]$, where $i$ is the smallest numbered leaf among its descendants, and $j$ is the largest. 
The leaf $i$ is therefore labelled by the interval $[i,i]$.

\begin{Lemma}\label{lem:catalan-bijection}
	There is a one-to-one correspondence between rooted binary trees with $k$ leaves and election processes on team hierarchies with $k$ layers.
\end{Lemma}
\begin{proof}
	Let $C$ be the \emph{left comb} on $k$ vertices, defined to be	the rooted binary tree whose internal vertices are labeled $[1,2], [1,3], \ldots, [1,k]$.
	For instance, the tree below is the left comb on $4$ vertices with non-leaf nodes marked $\bullet$
	and leaves marked $\circ$.
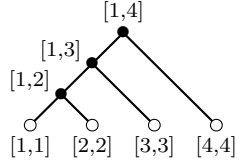
	\[
	\begin{tikzpicture}[x=1.5em, y=-1.5em, rotate=-45, thick]
		\draw (0,0) -- (3,0); \node[above] at (0,0) {$\scriptstyle [1,4]$};
		\draw (0,0) -- (0,3); \node[above, left, yshift=5pt] at (0,1) {$\scriptstyle [1,3]$};
		\draw (0,1) -- (2,1); \node[above, left, yshift=5pt] at (0,2) {$\scriptstyle [1,2]$};
		\draw (0,2) -- (1,2);
		\node [below] at (0,3) {$\scriptstyle [1,1]$};
		\node [below] at (1,2) {$\scriptstyle [2,2]$};
		\node [below] at (2,1) {$\scriptstyle [3,3]$};
		\node [below] at (3,0) {$\scriptstyle [4,4]$};
		\node at (0,0) {$\bullet$};
		\node at (0,1) {$\bullet$};
		\node at (0,2) {$\bullet$};
		\wcirc{0}{3}
		\wcirc{1}{2}
		\wcirc{2}{1}
		\wcirc{3}{0}

	\end{tikzpicture}
	\]
	It is well understood \cite{Tamari} that all other rooted binary trees can be generated from $C$ by successively applying \emph{right tree rotations} to its subtrees:
		\begin{equation}
		\label{eq:treeRotation}
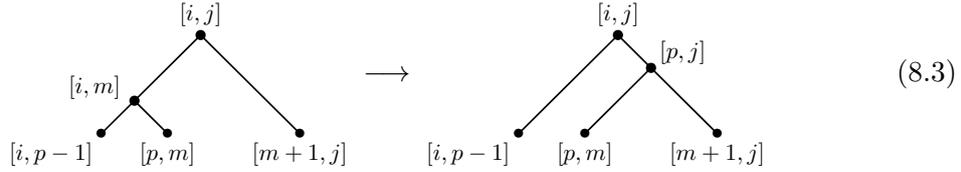
		\begin{adjustbox}{scale=0.8}
			\begin{tikzpicture}[x=2em, y=-2em, rotate=-45, thick, baseline=-2em]
				\filldraw (0,0) circle (2pt) node[anchor=south] {$[i,j]$} -- (3,0)  node[anchor=north] {$[m+1,j]$};
				\filldraw (0,0) -- (0,3)node[anchor=north east] {$[i,p-1]$};
				\filldraw (0,2) circle (2pt) node[anchor=south east, xshift=3pt, yshift=-3pt] {$[i,m]$~{\,}} -- (1,2);
				\filldraw (1,2) node[anchor=north] {$[p,m]$};
				\node at (0,3) {$\bullet$};
				\node at (1,2) {$\bullet$};
				\node at (3,0) {$\bullet$};
			\end{tikzpicture}
		\end{adjustbox}
		\longrightarrow 
		\begin{adjustbox}{scale=0.8}
			\begin{tikzpicture}[x=2em, y=-2em, rotate=-45, thick, baseline=-2em]
				\filldraw (0,0) circle (2pt) node[anchor=south] {$[i,j]$} -- (3,0)  node[anchor=north] {$[m+1,j]$};
				\filldraw (0,0) -- (0,3) node[anchor=north east] {$[i,p-1]$};
				\filldraw (1,0) circle (2pt) node[anchor=south west, yshift=-2pt] {$[p,j]$~{\,}} -- (1,2);
				\filldraw (1,2) node[anchor=north] {$[p,m]$};
				\node at (0,3) {$\bullet$};
				\node at (1,2) {$\bullet$};
				\node at (3,0) {$\bullet$};
			\end{tikzpicture}        
		\end{adjustbox}
	\end{equation}
	
	Let $T$ be a rooted binary tree with $k$ leaves, let $N_k > \cdots > N_2 > N_1 > 0$ be integers, and set $N_0 = 0$. We assign a binomial coefficient to each internal node of $T$ as follows. Let $[i,j]$ be an internal node, and let $\ell$
	be largest such that $i \leq \ell < j$ and $[i,\ell]$ is a vertex of $T$. (It is possible that $[i,\ell]$ is a
	leaf.) Then, $[i,j]$ is assigned to
	\[
	\binom{N_j - N_{i-1}}{N_{\ell}-N_{i-1}}.
	\]
	The tree $T$ is mapped to the product of the binomial coefficients assigned to its internal vertices.
	Since $N_0=0$,  the left comb $C$ is sent to
	\[
	\binom{N_k}{N_{k-1}} \cdots \binom{N_3}{N_2} \binom{N_2}{N_1}.
	\]
	The pairs of vertices highlighted in~\eqref{eq:treeRotation} are mapped to the left-hand and right-hand sides of~\eqref{eq:subsetOfSubsetsFlip}, respectively. Hence, the map is bijective.
\end{proof}

\begin{Example}
	We have the following correspondence:
	\[
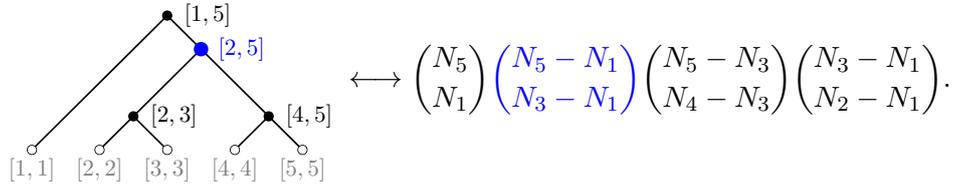
	\begin{adjustbox}{scale=0.8}
		\begin{tikzpicture}[x=2em, y=-2em, rotate=-45, thick, baseline=-3em]
			\filldraw (0,0) circle (2pt) node[anchor=west] {~$[1,5]$} -- (4,0);
			\filldraw (0,0) -- (0,4);
			\filldraw (1,0) circle (2pt) node[anchor=west, blue] {~$[2,5]$} -- (1,3);
			\filldraw (1,2) circle (2pt) node[anchor=west] {~$[2,3]$} -- (2,2);
			\filldraw (3,0) circle (2pt) node[anchor=west] {~$[4,5]$} -- (3,1);
			\filldraw[blue] (1,0) circle (3pt);
			\draw[gray] (0,4) node[anchor = north] {$[1,1]$};
			\draw[gray] (1,3) node[anchor = north] {$[2,2]$};
			\draw[gray] (2,2) node[anchor = north] {$[3,3]$};
			\draw[gray] (3,1) node[anchor = north] {$[4,4]$};
			\draw[gray] (4,0) node[anchor = north] {$[5,5]$};
		\wcirc{0}{4}
		\wcirc{1}{3}
		\wcirc{2}{2}
		\wcirc{3}{1}
		\wcirc{4}{0}
		\end{tikzpicture}
	\end{adjustbox}
	\longleftrightarrow \binom{N_5}{N_1} \textcolor{blue}{\binom{N_5-N_1}{N_3-N_1}} \binom{N_5-N_3}{N_4-N_3} \binom{N_3-N_1}{N_2-N_1}.
	\]
	For instance, consider the highlighted vertex $[2,5]$. The largest $\ell$ such that $[2,\ell]$ is a vertex, is $\ell=3$. Consequently, it corresponds to $\binom{N_5 - N_{2-1}}{N_{3} - N_{2-1}}$.
\end{Example}

\begin{Corollary}
	Given integers $N_k > \cdots > N_2 > N_1 > N_0 = 0$ in generic position, Corollary~\ref{cor:extGTL} yields $C_{k-1}$ isomorphic constructions of $\bigotimes_{i = 1}^{k-1} \Ext^{N_i} \Sym^{N_{i+1}-1} \E$.
\end{Corollary}

\begin{Example}
	The $\SL_2(\F)$-plethysms shown in the diagram below  % SL_2(F)-reps -> SL_2(F)-plethysms as in introduction
	each of dimension $\binom{15}{9}\binom{9}{5}\binom{5}{2}$, are pairwise isomorphic via successive applications of the isomorphism from Corollary~\ref{cor:extGTL}.
	
	\[
	\begin{adjustbox}{scale=0.8}
		\begin{tikzpicture}[rotate=20, scale=1.5]
			\def\radius{2cm}
			\node (E0) at ({90 + 0*72}:\radius) {
				$\Ext^{9}\Sym^{14}E\otimes\Ext^{5}\Sym^8E\otimes\Ext^{2}\Sym^4E$
			};
			\node (E1) at ({90 + 1*72}:\radius) {
				$\Ext^{5}\Sym^{14}E\otimes\Ext^{4}\Sym^9E\otimes\Ext^{2}\Sym^4E$
			};
			\node (E2) at ({90 + 2*72}:\radius) {
				$\Ext^{2}\Sym^{14}E\otimes\Ext^{4}\Sym^9E\otimes\Ext^{3}\Sym^{12}E$
			};
			\node (E3) at ({90 + 3*72}:\radius) {
				$\Ext^{2}\Sym^{14}E\otimes\Ext^{7}\Sym^{12}E\otimes\Ext^{3}\Sym^6E$
			};
			\node (E4) at ({90 + 4*72}:\radius) {
				$\Ext^{9}\Sym^{14}E\otimes\Ext^{2}\Sym^8E\otimes\Ext^{3}\Sym^6E$
			};
			
			% Draw pentagon
			\draw[thick, black, -stealth] (E0) -- (E1);
			\draw[thick, black, -stealth] (E1) -- (E2);
			\draw[thick, black, -stealth] (E0) -- (E4);
			\draw[thick, black, -stealth] (E4) -- (E3);
			\draw[thick, black, -stealth] (E3) -- (E2);
		\end{tikzpicture}
	\end{adjustbox}
	\]
      \end{Example}

      \section*{Acknowledgements}
      We thank Darij Grinberg for his comments to a previous version of this document. We thank Joseph Pappe for sharing with us a bijection between election processes and the vertices of the associahedron.

\providecommand{\bysame}{\leavevmode\hbox to3em{\hrulefill}\thinspace}
\providecommand{\MR}{\relax\ifhmode\unskip\space\fi MR }
\providecommand{\MRhref}[2]{%
	\href{http://www.ams.org/mathscinet-getitem?mr=#1}{#2}
}
\providecommand{\href}[2]{#2}

\end{document}